%% file: main.tex
\documentclass[11pt,reqno]{amsart}
\input{preamble.tex}
\title[Full mad families of vector spaces]{Full mad families of vector spaces and two local Ramsey theories}
\author{Clement Yung}

\begin{document}

\begin{abstract}
    Let $E$ be a vector space over a countable field of dimension $\aleph_0$. Two infinite-dimensional subspaces $V,W \subseteq E$ are \emph{almost disjoint} if $V \cap W$ is finite-dimensional. This paper provides some improvements on results about the definability of maximal almost disjoint families (mad families) of subspaces in \cite{S19}. We construct a full mad family of block subspaces in $\ZFC$, answering a problem by Smythe in the positive. A variant of this construction shows that there exists a completely separable mad family of block subspaces in $\ZFC$. We also discuss the abstract Mathias forcing introduced by Di Prisco-Mijares-Nieto in \cite{DMN15}, and apply it to show that in the Solovay's model obtained by the collapse of a Mahlo cardinal, there are no full mad families of subspaces over $\mathbb{F}_2$.
\end{abstract}

\maketitle

\section{Introduction}
\label{sec:intro}

Let $E$ be a vector space over a countable field $\mathbb{F}$ of dimension $\aleph_0$. In \cite{S19}, Smythe studied \emph{almost disjoint families} of infinite-dimensional subspaces $V,W \subseteq E$, i.e. families of infinite-dimensional subspaces of $E$ whose pairwise intersections are finite-dimensional, and raised the following questions.
\begin{prob}
\label{prob:definable.mad.families}
    Can mad families of subspaces be definable? In particular:
    \begin{enumerate}
        \item Are there analytic mad families of subspaces?
        
        \item Assuming the existence of large cardinals, are there mad families of subspaces in $\L(\R)$?
    \end{enumerate}
\end{prob}

The study of almost disjoint families in various linear algebraic notions, or applications of the theory of almost disjoint families to linear algebra or functional analysis, predates Smythe's work in \cite{S19}. In \cite{BS91}, Baumgartner-Spinas studied the existence of quadratic spaces, and used almost disjoint families in $[\omega]^\omega$ to show that under the assumption $\frak{b} = \aleph_1$, there exists an $\aleph_1$-dimensional ($**$)-space over a class of countable fields. Possible cardinalities of variants of almost disjoint families, including \emph{almost orthogonal families}, of infinite-dimensional subspaces of a separable Hilbert space were also studied in \cite{B11} and \cite{W08}. Kolman also studied almost disjoint families of infinite-dimensional subspaces with no additional topological assumptions in \cite{K00} (which he called \emph{almost disjoint packings}), but he assumed that almost disjoint families satisfy a property stronger than pairwise almost disjointness (see also footnote 2 of \cite{S19}).

Horowitz-Shelah proved in \cite{HS22} that there are no analytic mad families of subspaces over $\mathbb{F}_2$, but the first problem remains open for other fields. Smythe remarked in \cite{S19} that if there is a supercompact cardinal, then every set in $\L(\R)$ is $\H$-strategically Ramsey, so there are no full mad families of subspaces there (the definition of a \emph{full} mad family of subspaces will be introduced in the subsequent paragraph). Strategically Ramsey subsets of $E^{[\infty]}$ were introduced in \cite{R09} and \cite{R10} to study the combinatorics of Banach space theory and dichotomies, and Smythe introduced a localised version of the strategically Ramsey property in \cite{S18}. To the author's knowledge, the precise large cardinal strength needed for the non-existence of mad families of subspaces in $\L(\R)$ is mostly unknown.

Inspired by Mathias' original proof of the non-existence of analytic mad families in $[\omega]^\omega$ in \cite{M77}, Smythe approached this problem using a local Ramsey theoretic approach in \cite{S19}, replacing $\H$-Ramsey with $\H$-strategically Ramsey. Given an almost disjoint family $\A$ of vector spaces, he extended Mathias' definition of the ``happy family'' $\mathfrak{F}_\A$ of subsets of $\omega$, to a ``family'' $\H(\A)$ of subspaces of $E$. His approach requires $\H(\A)$ to be \emph{full}, a property which closely resembles the pigeonhole property of coideals of $[\omega]^\omega$. Calling a mad family $\A$ \emph{full} if $\H(\A)$ is full, Smythe proved the following theorem:

\begin{theorem}[Smythe \cite{S19}]
\label{thm:mad.not.analytic}
    There exist no analytic full mad families of subspaces.
\end{theorem}

Smythe proved this theorem by studying the combinatorial properties of families of \emph{block subspaces}. Block subspaces are subspaces with a basis of vectors with strictly increasing support (see \Sec\ref{sec:setup} for a precise definition). Block subspaces are more combinatorially tame, and since every infinite-dimensional subspace contains an infinite-dimensional block subspace, an almost disjoint family of subspaces is maximal with respect to block subspaces iff it is maximal. 

While it is clear from Smythe's proof that the fullness property plays an important role in understanding the definability of mad families of subspaces, the extent which mad families of subspaces satisfy this property remains poorly understood. In particular, Smythe asked the following question in \cite{S19}:

\begin{prob}[Smythe \cite{S19}]
\label{prob:full.mad.families}
    Does $\ZFC$ prove the existence of a full mad family of subspaces? Even more, does $\ZFC$ prove that every mad family of subspaces is full?
\end{prob}

In the same paper, Smythe proved that assuming $\frak{p} = \frak{c}$, there exists a full mad family of block subspaces.

Since every infinite-dimensional subspace contains an infinite-dimensional block subspace, not all generality is lost when restricting our attention to block subspaces. However, several open problems concerning differences between mad families of subspaces and of block subspaces remain. Consider the following two cardinals:
\begin{align*}
    \avec^* &:= \min\{|\A| : \text{$\A$ is an infinite mad family of subspaces over $\mathbb{F}$}\}, \\
    \avec &:= \min\{|\A| : \text{$\A$ is an infinite mad family of block subspaces over $\mathbb{F}$}\}.
\end{align*}
It is unknown if they are equal in $\ZFC$.

\begin{prob}[Smythe, \cite{S19}]
\label{prob:avec*.and.avec}
    Do we lose any generality when restricting mad families to block subspaces? In particular, is $\avec^* = \avec$ true in $\ZFC$?
\end{prob}

The first result in the paper answers the first part of Problem \ref{prob:definable.mad.families} in the positive.

\begin{theorem}
\label{thm:main.full.mad.family}
    There exists a full mad family of block subspaces.
\end{theorem}

Our method is inspired by a variant of the construction of a completely separable mad family in $[\omega]^\omega$ in \cite{MRS14} assuming $\frak{s} \leq \frak{a}$, provided by Guzm\'{a}n in \cite{G21}. Construction of completely separable mad families in $[\omega]^\omega$ has been extensively studied in the past two decades. Furthermore, complete separable mad families in $[\omega]^\omega$ have proved to be a powerful tool for constructing other set-theoretic objects. These include objects of interest in functional analysis, such as a masa in the Calkin algebra which does not lift to a masa of $B(H)$ in \cite{SS11}. Shelah attained a breakthrough in \cite{Sh11}, where he showed that a completely separable mad family exists assuming either $\frak{c} < \aleph_\omega$, or there is no inner model with a measurable cardinal. However, the existence of a completely separable mad family in $\ZFC$ alone remains open. We shall show that the natural generalisation of a completely separable mad family to vector spaces (see Definition \ref{def:cs.mad.family}) exists in $\ZFC$.

\begin{theorem}
\label{thm:cs.mad.family}
    There exists a completely separable mad family of block subspaces.
\end{theorem}

Our constructions for Theorem \ref{thm:main.full.mad.family} and \ref{thm:cs.mad.family} utilise various results from the local Ramsey theory of block subspaces developed by Smythe in \cite{S18} and \cite{S19}, along with a few additional results that we shall prove. The non-requirement of additional set-theoretic hypothesis is a result of the $\ZFC$ inequality $\non(\M) \leq \avec$ by Smythe in \cite{S19}. Since $\frak{s} \leq \non(\M)$ holds in $\ZFC$, the vector space variant of the inequality $\frak{s} \leq \frak{a}$, i.e. $\frak{s} \leq \avec$, is also a $\ZFC$ inequality. Readers may also refer to \cite{Rag09} and \cite{Rag10} for more constructions of mad families in $[\omega]^\omega$ involving the cardinal invariant $\non(\M)$. 

The second part in this paper addresses the existence of full mad families of block subspaces over $\mathbb{F}_2$ in Solovay's model. We utilise the abstract Mathias forcing developed by Di Prisco--Mijares--Nieto in \cite{DMN15} --- the abstract Mathias forcing was developed to expand the local Ramsey theory of $[\omega]^\omega$ to topological Ramsey spaces (i.e. closed triples $(\cal{R},\leq,r)$ satisfying the axioms \textbf{A1}-\textbf{A4} introduced by Todor\v{c}evi\'{c} in \cite{T10}). 

If $E$ is a countable vector space over $\mathbb{F}_2$, then the pigeonhole principle holds, so $(E^{[\infty]},\leq,r)$ is a topological Ramsey space. By identifying vectors in $E$ with their supports, this topological Ramsey space coincides with Hindman's space $(\FIN^{[\infty]},\leq,r)$ studied in \cite{B87}. Filters in $\FIN^{[\infty]}$ have been studied by Blass in \cite{B87}, and are known as \emph{ordered-union ultrafilters}. Additional overlaps between Blass' theory of ordered-union ultrafilters and Smythe's theory of filters in \cite{S18} and \cite{S23} include the notion of \emph{stable} ordered-union ultrafilters, which in Smythe's terminology is known as a full (p)-filter. This allows us to study the local Ramsey theory of infinite block sequences in $E$ from the following perspectives:
\begin{enumerate}
    \item Abstract local Ramsey theory by Di Prisco--Mijares--Nieto in \cite{DMN15}.
    
    \item Local Ramsey theory of block sequences by Smythe in \cite{S18} and \cite{S23}.
    
    \item The theory of (stable) ordered-union ultrafilters by Blass in \cite{B87}.
\end{enumerate}
In particular, the first two theories proposed versions of (semi)coideals and (ultra)filters in $E^{[\infty]}$, along with definitions of various selectivity properties. In \Sec\ref{sec:TRS.and.abstract.mathias.forcing} and \Sec\ref{sec:TRS.block.subspaces}, we shall show that these three theories coincide, so definitions may be used interchangeably and results proven in the individual theories may be freely applied to (semi)coideals and (ultra)filters in $E^{[\infty]}$.

Let $\Coll(\omega,<\!\kappa)$ be the L\'{e}vy collapse of uncountable cardinal $\kappa$ to $\omega_1$. Using the abstract Mathias forcing developed, we may apply the abstract Mathias forcing to obtain the following variant of a celebrated theorem of Mathias in \cite{M77}:
\begin{theorem}
\label{thm:mathias.thm.vector.spaces}
    Let $\kappa$ be a Mahlo cardinal, and let $G$ be a $\Coll(\omega,<\!\kappa)$-generic filter. Suppose that $E$ is a countable vector space over $\mathbb{F}_2$. If $\H \in \V[G]$ is a selective coideal on $E^{[\infty]}$, and $\X \in \L(\R)^{\V[G]}$ is a subset of $E^{[\infty]}$, then $\X$ is $\H$-Ramsey.
\end{theorem}

A rough sketch of the proof of Theorem \ref{thm:mathias.thm.vector.spaces} (in the context of $\FIN$) was stated by Blass in \cite{B87}, and he stated that it ``will be shown, in a more general context'' in \cite{B88}, but the published version of \cite{B88} does not contain such a proof. 

As a consequence, we obtain the following corollary:
\begin{corollary}
\label{cor:no.full.mad.family.in.L(R)}
    Let $\kappa$ be a Mahlo cardinal, and let $G$ be a $\Coll(\omega,<\!\kappa)$-generic filter. Then there are no full mad families of subspaces over $\mathbb{F}_2$ in $\L(\R)^{\V[G]}$.
\end{corollary}

Our paper is organised as follows:
\begin{enumerate}
    \item[(\Sec\ref{sec:setup})] We provide a concise introduction for various definitions and notations to block subspaces of $E$. We prove that finite intersections of block subspaces are also block subspaces (Lemma \ref{lem:block.subspaces.intersection.lemma}).
    
    \item[(\Sec\ref{sec:local.ramsey.theory})] We review the local Ramsey theory of block subspaces developed by Smythe, i.e. combinatorial properties of \emph{semicoideals} of block subspaces (which Smythe calls \emph{families}). We introduce the notion of a \emph{selective} semicoideal of block subspaces, a strengthening of the (p)-property. 
    
    \item[(\Sec\ref{sec:ad.families})] We introduce some definitions and basic results about almost disjoint families of subspaces. We show that $\avec^*$ is uncountable, addressing a remark by Smythe in \cite{S19}. We also show that the semicoideal $\H^-(\A)$ is selective for any almost disjoint family $\A$ of subspaces --- the non-requirement for $\A$ to be maximal is crucial in the proofs of Theorems \ref{thm:main.full.mad.family} and \ref{thm:cs.mad.family}.
    
    \item[(\Sec\ref{sec:full.mad.families})] We apply various lemmas and results proven in \Sec\ref{sec:local.ramsey.theory} and \Sec\ref{sec:ad.families} to provide detailed proofs of Theorems \ref{thm:main.full.mad.family} and \ref{thm:cs.mad.family}.
    
    \item[(\Sec\ref{sec:TRS.and.abstract.mathias.forcing})] We provide an overview of the axioms of topological Ramsey spaces \textbf{A1}-\textbf{A4} introduced by Todor\v{c}evi\'{c}, and the abstract local Ramsey theory developed in \cite{DMN15}. We show that given a closed triple $(\cal{R},\leq,r)$ satisfying \textbf{A1}-\textbf{A4}, if $\H \subseteq \cal{R}$ is semiselective then $\mathbb{M}_\H$ satisfies the Mathias property and the Prikry property. This was proven in \cite{DMN15}, but it was assumed that the almost reduction relation $\leq^*$ is transitive (which is not generally true, such as for the topological Ramsey space of strong subtrees). We provide a proof of this result without this assumption.
    
    \item[(\Sec\ref{sec:TRS.block.subspaces})] We study the intersection between the local Ramsey theory of block subspaces and abstract local Ramsey theory when considering block subspaces over $\mathbb{F}_2$. We show that various notions and results introduced in \Sec3 and \Sec6 coincide in the setting of block subspaces over $\mathbb{F}_2$ and provide a proof of Theorem \ref{thm:mathias.thm.vector.spaces} and Corollary \ref{cor:no.full.mad.family.in.L(R)}.
    
    \item[(\Sec\ref{sec:conclusion})] We conclude the paper with several further remarks and open problems.
\end{enumerate}

\section{Vector spaces and block subspaces}
\label{sec:setup}
We introduce some basic definitions and properties of vector spaces and block subspaces. Throughout this article, only non-zero vectors shall be considered (thus, for instance, for two subspaces $V,W$, we write $V \cap W = \emptyset$ if their intersection is the trivial subspace).

\subsection{Setup and notations}
Let $E$ be a vector space over a countable field $\mathbb{F}$, with a countable Hamel basis $(e_n)_{n<\omega}$. For each $x \in E$, we may write $x := \sum_{n<\omega} \lambda_n(x)e_n$, where $\lambda_n(x) \in \mathbb{F}$. This allows us to define the \emph{support} of a vector $x \in E$ by:
\begin{align*}
    \supp(x) := \{n < \omega : \lambda_n(x) \neq 0\}.
\end{align*}
Note that $\supp(x)$ is finite for all $x \in E$, as $(e_n)_{n<\omega}$ is a Hamel basis. This gives us a partial order $<$ on $E$, where for any two vectors $x,y \in E$:
\begin{align*}
    x < y \iff \max(\supp(x)) < \min(\supp(y)).
\end{align*}
A \emph{block sequence} of vectors in $E$ is thus a $<$-increasing sequence. Note that any block sequence of vectors is linearly independent. 

\begin{definition}
    An infinite-dimensional subspace $V \subseteq E$ is a \emph{block subspace} if $V$ has a \emph{block basis}, i.e. an infinite block sequence $(x_n)_{n<\omega}$ such that $V = \func{span}\{x_n : n < \omega\}$.
\end{definition}

Since such a block basis is unique up to scaling, we may conflate block subspaces of $E$ with infinite block sequences. This allows us to equip $E^{[\infty]}$ with a Polish topology by considering the first difference metric on infinite block sequences --- that is, if $(x_n)_{n<\omega}$ and $(y_n)_{n<\omega}$ are infinite block sequences, then:
\begin{align*}
    d((x_n)_{n<\omega},(y_n)_{n<\omega}) = 2^{-\min\{n : x_n \neq y_n\}}.
\end{align*}

\begin{notation}
\label{not:notations}
\hfill
    \begin{enumerate}
        \item We let $E^{[\infty]}$ (resp. $E^{[<\infty]}$) denote the set of infinite (resp. finite) block sequences of vectors in $E$. 

        \item $E^{[\leq\infty]} := E^{[\infty]} \cup E^{[<\infty]}$.

        \item If $A = (x_n)_{n<\omega} \in E^{[\infty]}$ and $N < \omega$, we write $A/N := (x_n)_{n \geq N}$\footnote{This notation differs from that introduced in \cite{S23}, where it was instead defined to be the tail sequence of $A$ consisting of vectors supported above $N$.}.
        
        \item If $A = (x_n)_{n<\omega} \in E^{[\infty]}$ and $n < \omega$, we write $r_n(A) := (x_i)_{i<n}$.
        
        \item If $a = (x_n)_{n<N},b = (y_m)_{m<M} \in E^{[<\infty]}$, we write $a < b$ if $x_{N-1} < y_0$.
        
        \item If $A \in E^{[\infty]}$ and $a \in E^{[<\infty]}$, we write $a < A$ if $a < r_1(A)$.
        
        \item If $A \in E^{[\infty]}$ and $a \in E^{[<\infty]}$ we write $d_A(a) := \min\{N < \omega : a < A/N\}$.

        \item If $A \in E^{[\infty]}$ and $a \in E^{[<\infty]}$, then $A/a := A/d_A(a)$.
        
        \item If $Y \subseteq E$, we write:
        \begin{align*}
            \c{Y} := \func{span}(Y).
        \end{align*}

        \item If $(x_n)_{n<\alpha} \in E^{[\leq\infty]}$, we write:
        \begin{align*}
            \c{(x_n)_{n<\alpha}} := \c{x_n : n < \alpha}.
        \end{align*}

        \item If $A,B \in E^{[\infty]}$, we write $A \leq B$ iff $\c{A} \subseteq \c{B}$.
        
        \item If $A,B \in E^{[\infty]}$, we write $A \leq^* B$ iff $A/N \leq B$ for some $N < \omega$.
        
        \item If $A \in E^{[\infty]}$, we write $E^{[\infty]}\restrictedto A := \{B \in E^{[\infty]} : B \leq A\}$.
        
        \item If $A \in E^{[\infty]}$, we write $E^{[<\infty]}\restrictedto A := \{a \in E^{[<\infty]} : \c{a} \subseteq \c{A}\}$.
    \end{enumerate}
\end{notation}

Block sequences/subspaces were introduced as they are more combinatorially well-behaved than arbitrary infinite-dimensional subspaces. The focus on block sequences is further justified by the observation that every infinite-dimensional subspace of $E$ contains a block subspace. This will be apparent once we have introduced the relevant definitions and notations for almost disjoint families of subspaces.

\subsection{Intersection of block subspaces}
Given $A,B \in E^{[\infty]}$, if $\c{A} \cap \c{B}$ is infinite-dimensional, then by the remark above it contains a block subspace. The following lemma shows that $\c{A} \cap \c{B}$ is itself a block subspace\footnote{This lemma corrects footnote 8 of \cite{S19}, which incorrectly stated that intersection of block subspaces need not be block.}. This observation is central in ensuring that the proof of Theorem \ref{thm:main.full.mad.family} and \ref{thm:cs.mad.family} do not require additional set-theoretic hypotheses.

\begin{lemma}
\label{lem:block.subspaces.intersection.lemma}
    Let $A,B \in E^{[\infty]}$ be such that $\c{A} \cap \c{B}$ is infinite-dimensional. Then there exists some $C \in E^{[\infty]}$ such that $\c{C} = \c{A} \cap \c{B}$.
\end{lemma}

\begin{proof}
    We write $A = (x_n)_{n<\omega}$ and $B = (y_n)_{n<\omega}$. For each $v \in \c{A} \cap \c{B}$, we let:
    \begin{align*}
        \supp_A(v) &:= \{n < \omega : \supp(x_n) \subseteq \supp(v)\}, \\
        \supp_B(v) &:= \{n < \omega : \supp(y_n) \subseteq \supp(v)\}.
    \end{align*}
    We remark that for all $v,w \in \c{A} \cap \c{B}$:
    \begin{align*}
        \supp(v) \subseteq \supp(w) &\iff \supp_A(v) \subseteq \supp_A(w) \\
        &\iff \supp_B(v) \subseteq \supp_B(w).
    \end{align*}
    For an arbitrary $v \in E$ and $k \in \supp(v)$, recall that $\lambda_k(v) \in \mathbb{F}$ denotes the unique coefficient such that $v = u + \lambda_k(v)e_k + w$ for some vectors $u < e_k < w$ (or $u = 0$ or $w = 0$).

    \begin{claim}
    \label{claim:supp(v).intersection}
        Let $v_0,v_1 \in \c{A} \cap \c{B}$, and suppose that $\supp_A(v_0) \cap \supp_A(v_1) \neq \emptyset$. Then there exists some $w \in \c{A} \cap \c{B}$ such that $\supp_A(w) = \supp_A(v_0) \cap \supp_A(v_1)$.
    \end{claim}

    \begin{midproof}
        Let $v_0,v_1 \in \c{A} \cap \c{B}$ be such that $\supp_A(v_0) \cap \supp_A(v_1) \neq \emptyset$. We may write:
        \begin{align}
            v_0 &= \sum_{i \in \supp_A(v_0)} \xi_i^0x_i = \sum_{j \in \supp_B(v_0)} \mu_j^0y_j, \label{eq:v_0} \\
            v_1 &= \sum_{i \in \supp_A(v_1)} \xi_i^1x_i = \sum_{j \in \supp_B(v_1)} \mu_j^1y_j. \label{eq:v_1}
        \end{align}
        where $\xi_i^0,\xi_i^1,\xi_j^0,\xi_j^1$ are all non-zero. We let:
        \begin{align*}
            w := \sum_{i \in \supp_A(v_0) \cap \supp_A(v_1)} \xi_i^0 x_i, \\
            w' := \sum_{j \in \supp_B(v_0) \cap \supp_B(v_1)} \mu_j^0 y_j.
        \end{align*}
        It's clear that $\supp_A(w) = \supp_A(v_0) \cap \supp_A(v_1)$. We shall show that $w = w'$. This implies that that $w \in \c{A} \cap \c{B}$, proving the claim.

        By the symmetry of the argument, it suffices to show that $\supp(w) \subseteq \supp(w')$, and $\lambda_k(w) = \lambda_k(w')$ for all $k \in \supp(w)$. Let $k \in \supp(w)$, so $k \in \supp(x_i)$ for some $i \in \supp_A(v_0) \cap \supp_A(v_1)$. By the second equalities in (\ref{eq:v_0}) and (\ref{eq:v_1}), we have that $k \in \supp(y_j)$ for some $j \in \supp_B(v_0)$, and $k \in \supp(y_{j'})$ for some $j' \in \supp_B(v_1)$. Since $B$ is a block sequence, $k \in \supp(y_j) \cap \supp(y_{j'})$ implies that $j = j'$. Therefore, $j \in \supp_B(v_0) \cap \supp_B(v_1)$, so $k \in \supp(w')$. By (\ref{eq:v_0}), we can conclude that:
        \begin{align*}
            \lambda_k(w) = \xi_i^0\lambda_k(x_i) = \lambda_k(v_0) = \mu_j^0\lambda_k(y_j) = \lambda_k(w'),
        \end{align*}
        which completes the proof. 
    \end{midproof}
    
    Let $K \subseteq \omega$ be the set of all $n$ such that $n \in \supp_A(v)$ for some $v \in \c{A} \cap \c{B}$. Since $\c{A} \cap \c{B}$ is infinite-dimensional, $K$ is infinite. Define a relation $\sim$ on $K$ such that for all $m,n \in K$:
    \begin{align*}
        m \sim n \iff \forall v \in \c{A} \cap \c{B}(m \in \supp_A(v) \leftrightarrow n \in \supp_A(v)).
    \end{align*}
    It's clear from the definition that $\sim$ is an equivalence relation, and that every equivalence class is finite. 

    \begin{claim}
        For each $n \in K$, there exists some $v \in \c{A} \cap \c{B}$ such that $\supp_A(v) = [n]_\sim$. 
    \end{claim}

    \begin{midproof}
        It follows from the definition of $\sim$ that:
        \begin{align*}
            [n]_\sim = \bigcap\{\supp_A(u) : u \in \c{A} \cap \c{B} \text{ and } n \in \supp_A(u)\}.
        \end{align*}
        However, since $\supp_A(u)$ is finite for each $u \in \c{A} \cap \c{B}$, we may find $u_0,\dots,u_{m-1} \in \c{A} \cap \c{B}$ such that $[n]_\sim = \bigcap_{i<m} \supp_A(u_i)$. Then the existence of such a $v$ follows from Claim \ref{claim:supp(v).intersection}.
    \end{midproof}

    Thus, for each $n$ we let $z_n \in \c{A} \cap \c{B}$ be a fixed vector such that $\supp_A(z_n) = [n]_\sim$. 

    \begin{claim}
    \label{claim:eq.classes.are.blocks}
        Let $m,n,l \in K$ be such that $m < n < l$, and suppose that $m \sim l$. Then $m \sim n \sim l$.
    \end{claim}

    \begin{midproof}
        Suppose otherwise, so there exist $m < n < l$ in $K$ such that $m \sim l$ but $n \notin [m]_\sim$. We write:
        \begin{align*}
            z_m &= \sum_{i \in \supp_A(z_m)} \xi_ix_i = \sum_{j \in \supp_B(z_m)} \mu_jy_j,
        \end{align*} 
        where $\xi_i,\mu_j$ are non-zero. Let $i_0 := \max\{i \in \supp_A(z_m) : i < n\}$ and $i_1 := \min\{i \in \supp_A(z_m) : i > n\}$, and let $k_0 := \max(\supp(x_{i_0}))$ and $k_1 := \min(\supp(x_{i_1}))$. We define:
        \begin{align*}
            w := \sum_{\substack{i \in \supp_A(z_m) \\ \max(\supp(x_i)) \leq k_0}} \xi_ix_i.
        \end{align*}
        We shall show that $w \in \c{A} \cap \c{B}$. This gives us the required contradiction, as $m \in \supp_A(w)$ but $l \not\in \supp_A(w)$. 

        Observe that for each $j \in \supp_B(z_m)$, either $k_0 < \min(\supp(y_j))$ or $\max(\supp(y_j)) < k_1$ --- otherwise, if $j \in \supp_B(z_m)$ is such that $\min(\supp(y_j)) \leq k_0 < k_1 \leq \max(\supp(y_j))$, then by the block sequence property of $B$, we have that for all $j' \neq j$, either $\max(\supp(y_{j'})) < k_0$ or $k_1 < \min(\supp(y_{j'}))$. Since $k_0 < \min(\supp(x_n)) \leq \max(\supp(x_n)) < k_1$ and $j \notin \supp_B(z_n)$, this implies that:
        \begin{align*}
            \supp(x_n) \cap \bigcup_{j' \in \supp_B(z_n)} \supp(y_{j'}) = \emptyset,
        \end{align*}
        which contradicts that $\supp(x_n) \subseteq \supp(z_n)$. 

        Note that for all $k \in \supp(z_m)$, either $k \leq k_0$ or $k_1 \leq k$. Thus, for all $j \in \supp_B(z_m)$, since $\supp(y_j) \subseteq \supp(z_m)$, if $\max(\supp(y_j)) < k_1$ then $\max(\supp(y_j)) \leq k_0$. Similarly, if $k_0 < \min(\supp(y_j))$ then $k_1 \leq \min(\supp(y_j))$. Therefore, we have that:
        \begin{align*}
            z_m = \sum_{\substack{j \in \supp_B(z_m) \\ \max(\supp(y_i)) \leq k_0}} \mu_j y_j + \sum_{\substack{j \in \supp_B(z_m) \\ k_1 \leq \min(\supp(y_i))}} \mu_j y_j.
        \end{align*}
        We may now express $\supp(z_m)$ in two different ways:
        \begin{align*}
            \bigcup_{i \in \supp_A(z_m)} \supp(x_i) &= \bigcup_{\substack{i \in \supp_A(z_m) \\ \max(\supp(x_i)) \leq k_0}} \supp(x_i) \cup \bigcup_{\substack{i \in \supp_A(z_m) \\ k_1 \leq \min(\supp(x_i))}} \supp(x_i), \\
            \bigcup_{j \in \supp_B(z_m)} \supp(y_j) &= \bigcup_{\substack{j \in \supp_B(z_m) \\ \max(\supp(y_j)) \leq k_0}} \supp(y_j) \cup \bigcup_{\substack{j \in \supp_B(z_m) \\ k_1 \leq \min(\supp(y_j))}} \supp(y_j).
        \end{align*}
        This implies that:
        \begin{align*}
            \bigcup_{\substack{i \in \supp_A(z_m) \\ \max(\supp(x_i)) \leq k_0}} \supp(x_i) &= \bigcup_{\substack{j \in \supp_B(z_m) \\ \max(\supp(y_j)) \leq k_0}} \supp(y_j), \\
            \bigcup_{\substack{i \in \supp_A(z_m) \\ k_1 \leq \min(\supp(x_i))}} \supp(x_i) &= \bigcup_{\substack{j \in \supp_B(z_m) \\ k_1 \leq \min(\supp(y_j))}} \supp(y_j).
        \end{align*}
        Therefore, if we let:
        \begin{align*}
            w' := \sum_{\substack{j \in \supp_B(z_m) \\ \max(\supp(y_j)) \leq k_0}} \mu_jy_j,
        \end{align*}
        then we have that $\supp(w) = \supp(w')$. Furthermore, for any $k \in \supp(w)$, $\lambda_k(w) = \lambda_k(z_m) = \lambda_k(w')$, so $w = w'$. It's clear from the definition that $w \in \c{A}$ and $w' \in \c{B}$, so we have that $w \in \c{A} \cap \c{B}$, as desired.
    \end{midproof}

    We now choose an increasing sequence $(n_i)_{i<\omega}$ of representatives of all equivalence classes in $\sim$ (i.e. $\bigcup_{i<\omega} [n_i]_\sim = K$). We let $C := (z_{n_i})_{i<\omega}$. By Claim \ref{claim:eq.classes.are.blocks}, $C$ is an infinite block sequence. We may finish the proof of this lemma by proving the following claim:

    \begin{claim}
        $\c{C} = \c{A} \cap \c{B}$.
    \end{claim}

    \begin{midproof}
        Since $z_{n_i} \in \c{A} \cap \c{B}$ for all $i$, we have that $\c{C} \subseteq \c{A} \cap \c{B}$. Conversely, let $v \in \c{A} \cap \c{B}$. Let $n_{i_0},\dots,n_{i_{t-1}}$ be an increasing enumeration of all $n_{i_j}$'s such that $n_{i_j} \in \supp_A(v)$. Fix any $k_j \in \supp(x_{n_{i_j}}) \subseteq \supp(z_{n_{i_j}})$, and let:
        \begin{align*}
            \xi_j := \frac{\lambda_{k_j}(v)}{\lambda_{k_j}(z_{n_{i_j}})}.
        \end{align*}
        Note that $\xi_j \neq 0$ for all $j$. Now let:
        \begin{align*}
            v' := \sum_{j<t} \xi_jz_{n_{i_j}}.
        \end{align*}
        Clearly $v' \in \c{C}$. We shall show that $v = v'$, which proves the claim.

        We first note that $\supp_A(v') = \supp_A(v)$ (which implies that $\supp(v') = \supp(v)$) --- observe that $\supp_A(v') = \bigcup_{j<t} \supp_A(z_{n_{i_j}})$. If $m \in \supp_A(v)$, then $[m]_\sim \subseteq \supp_A(v)$ and $[m]_\sim = [n_{i_j}]_\sim = \supp_A(z_{n_{i_j}})$ for some $j$, so $m \in \supp_A(v')$. Conversely, if $m \in \supp_A(z_{n_{i_j}})$, then since $n_{i_j} \in \supp_A(v)$, we have that $[m]_\sim = [n_{i_j}]_\sim \subseteq \supp_A(v)$, so $m \in \supp_A(v)$. 

        It remains to show that $\lambda_k(v') = \lambda_k(v)$ for all $k \in \supp(v)$. We fix some $j < t$, and observe that:
        \begin{align*}
            \lambda_{k_j}(v') = \frac{\lambda_{k_j}(v)}{\lambda_{k_j}(z_{n_{i_j}})} \cdot \lambda_{k_j}(z_{n_{i_j}}) = \lambda_{k_j}(v).
        \end{align*}
        Since $v \in \c{A}$ and $k_j \in \supp(x_{n_{i_j}}) \subseteq \supp(v)$, there exist some vectors $u,w \in \c{A}$ (or $u = 0$ or $w = 0$) such that $u < x_{n_{i_j}} < w$, and:
        \begin{align*}
            v = u + \frac{\lambda_{k_j}(v)}{\lambda_{k_j}(x_{n_{i_j}})}x_{n_{i_j}} + w.
        \end{align*}
        Similarly, there exist some vectors $u',w' \in \c{A}$ such that $u' < x_{n_{i_j}} < w'$ (or $u' = 0$ or $w'      = 0$), and:
        \begin{align*}
            v' = u' + \frac{\lambda_{k_j}(v)}{\lambda_{k_j}(x_{n_{i_j}})}x_{n_{i_j}} + w'.
        \end{align*}
        Therefore, $v - v' \in \c{A} \cap \c{B}$, and $n_{i_j} \notin \supp(v - v')$. But this implies that $[n_{i_j}]_\sim \cap \supp(v - v') = \emptyset$, which is only possible if for all $k \in \supp(z_{n_{i_j}})$, $\lambda_k(v') = \lambda_k(v)$. Therefore, for all $j < t$ and $k \in \supp(z_{n_{i_j}})$, $\lambda_k(v') = \lambda_k(v)$, so $v' = v$, as desired.
    \end{midproof}
\end{proof}

\section{Local Ramsey theory of block subspaces}
\label{sec:local.ramsey.theory}
Smythe explored the local Ramsey theory of block subspaces of a countable vector space in \cite{S18}, mainly focusing on $\leq^*$-upward closed subsets of $E^{[\infty]}$ (which Smythe calls \emph{families}) and their various properties.  We provide an overview of this theory, along with several additional results which will prove to be useful in later sections.

\subsection{Semicoideals}
To avoid confusion with other uses of the word ``family'' such as almost disjoint families, we shall use the term ``semicoideals'' in place of ``families''.

\begin{definition}
\label{def:semicoideal}
    A subset $\H \subseteq E^{[\infty]}$ is a \emph{semicoideal} if it is $\leq^*$-upward closed. That is, if $A \in \H$ and $A \leq^* B$, then $B \in \H$.
\end{definition}

To cope with the failure of the pigeonhole principle in the setting of countable vector spaces over $\mathbb{F} \neq \mathbb{F}_2$, Smythe introduced a weaker form of the pigeonhole property of semicoideals, which he calls \emph{fullness}.

\begin{notation}
    Let $\H \subseteq E^{[\infty]}$ be a semicoideal. Given $A \in \H$, we write:
    \begin{align*}
        \H\restrictedto A := \{B \in \H : B \leq A\}.
    \end{align*}
\end{notation}

\begin{definition}
\label{def:full.semicoideal}
    Let $\H \subseteq E^{[\infty]}$ be a semicoideal.
    \begin{enumerate}
        \item A subset $\D \subseteq E^{[\infty]}$ is \emph{$\H$-dense} below some $A \in \H$ if for all $B \in \H\restrictedto A$, there exists some $C \leq B$ such that $C \in \D$. A set $Y \subseteq E$ is \emph{$\H$-dense} below $A$ if the set $\{B \in E^{[\infty]} : \c{B} \subseteq Y\}$ is.

        \item We say that $\H$ is \emph{full} if for all $Y \subseteq E$ such that $Y$ is $\H$-dense below some $A \in \H$, there exists some $B \in \H\restrictedto A$ such that $\c{B} \subseteq Y$.
    \end{enumerate}
\end{definition}

It turns out that if $E$ is a vector space over $\mathbb{F}_2$, then fullness coincides with the local pigeonhole principle property. In this case, we say that $\H$ is a \emph{coideal}.

\begin{definition}
\label{def:coideal.vector.spaces}
    Let $\H \subseteq E^{[\infty]}$ be a semicoideal. We say that $\H$ is a \emph{coideal} if for all $Y \subseteq E$ and $A \in \H$, there exists some $B \in \H\restrictedto A$ such that $\c{B} \subseteq Y$ or $\c{B} \subseteq Y^c$.
\end{definition}

By Hindman's theorem, $E^{[\infty]}$ is a coideal if $E$ is a vector space over $\mathbb{F}_2$.

\begin{lemma}
\label{lem:full.iff.coideal}
    Suppose that $E$ is a vector space over $\mathbb{F}_2$, and let $\H \subseteq E^{[\infty]}$ be a semicoideal. Then $\H$ is full iff $\H$ is a coideal.
\end{lemma}

\begin{proof}
    Suppose that $\H$ is a coideal. Let $Y \subseteq E$, and suppose that $Y$ is $\H$-dense below some $A \in \H$. Since $\H$ is a coideal, there exists some $B \in \H\restrictedto A$ such that $\c{B} \subseteq Y$ or $\c{B} \subseteq Y^c$. Since $Y$ is $\H$-dense below $A$, the latter option is not possible.

    Now suppose that $\H$ is full. Let $Y \subseteq E$, and let $A \in \H$. If $Y$ is $\H$-dense below $A$, then we are done by the fullness of $\H$, so assume otherwise. This means that there exists some $B \in \H\restrictedto A$ such that for all $C \leq B$, $\c{C} \not\subseteq Y$. By Hindman's theorem, for all $C \in \H\restrictedto B$, there exists some $D \leq C$ such that $\c{D} \subseteq Y^c$ (as $\c{D} \not\subseteq Y$). But this means that $Y^c$ is $\H$-dense below $B$, so there exists some $C \in \H\restrictedto B$ such that $\c{C} \subseteq Y^c$, as desired.
\end{proof}

\subsection{Selectivity properties}
In the process of studying semicoideals related to mad families (particularly $\H(\A)$, see Definition \ref{def:ad.semi-ideals.and.semicoideals}), Smythe extended several local Ramsey theoretic properties of coideals of $[\N]^\infty$ to semicoideals of $E^{[\infty]}$. Selectivity properties of (ultra)filters (see Definition \ref{def:filter}) in $E^{[\infty]}$ has also been studied in \cite{S23}.  The first such property extends the (p)-property for coideals of $[\N]^\infty$ (see Lemma 7.4, \cite{T10}).

\begin{definition}
\label{def:(p)-semicoideal}
    Let $\H \subseteq E^{[\infty]}$ be a semicoideal.
    \begin{enumerate}
        \item Let $(A_n)_{n<\omega}$ be a $\leq$-decreasing sequence in $\H$. We say that $B \in \H$ \emph{weakly diagonalises} $(A_n)_{n<\omega}$ if $B \leq A_0$ and $B \leq^* A_n$ for all $n$\footnote{In Definition 2.2 of \cite{S18}, this property was similarly defined using the term \emph{diagonalises}. We shall reserve this term for a stronger form of diagonalisation later (Definition \ref{def:selective})}.

        \item We say that $\H$ satisfies the \emph{(p)-property}, or that $\H$ is a \emph{(p)-semicoideal}, if every $\leq$-decreasing sequence in $\H$ has a weak diagonalisation in $\H$.
    \end{enumerate}
\end{definition}

We shall provide a similar extension of the \emph{(q)-property} (see Lemma 7.4, \cite{T10}). Smythe also made a similar extension, which he termed as \emph{spreadness} (Definition 8.11, \cite{S18}).

\begin{definition}
\label{def:(q)-semicoideal}
    Let $\H \subseteq E^{[\infty]}$ be a semicoideal. We say that $\H$ satisfies the \emph{(q)-property}, or that $\H$ is a \emph{(q)-semicoideal}, if for all $A \in \H$ and for all increasing sequence of finite intervals $(I_m)_{m<\omega}$, there exists some $B \in \H\restrictedto A$ such that for all $n$, there exists some $m$ such that:
    \begin{enumerate}
        \item $\c{r_n(B)} \subseteq \c{r_{\min(I_m)-1}(A)}$, and;

        \item $B/n \leq A/\max(I_m)$.
    \end{enumerate}
\end{definition}

In a similar vein, we extend the definition of selectivity of coideals of $[\N]^\infty$ (see Definition 7.3, \cite{T10}) to semicoideals of block subspaces.

\begin{definition}
\label{def:selective}
    Let $\H \subseteq E^{[\infty]}$ be a semicoideal.
    \begin{enumerate}
        \item Let $A \in \H$, and let $(A_n)_{n<\omega}$ be a $\leq$-decreasing sequence in $\H$ below $A$ (i.e. $A_0 \leq A$). We say that $B \in \H\restrictedto A$ \emph{diagonalises $(A_n)_{n<\omega}$ below $A$} if $B/n \leq A_{d_A(r_n(B))}$ for all $n < \omega$.

        \item We say that $\H$ is \emph{selective} if for all $A \in \H$ and $\leq$-decreasing chain $(A_n)_{n<\omega}$ of elements in $\H$ below $A$, there exists a diagonalisation of $(A_n)_{n<\omega}$ below $A$ in $\H$. 
    \end{enumerate}
\end{definition}

Smythe also proposed a notion of \emph{strong diagonalisation} in Definition 4.1 of \cite{S18}, followed by a notion of \emph{strong (p)-property}. The following lemma shows that they are equivalent if $E$ is a vector space over a finite field. It is unclear if they are equivalent in case the underlying field is infinite.

\begin{lemma}
    Suppose that $E$ is a vector space over a finite field. Let $\H \subseteq E^{[\infty]}$ be a semicoideal. The following are equivalent:
    \begin{enumerate}
        \item $\H$ is selective.
        
        \item If $(A_a)_{a \in E^{[<\infty]}}$ generates a filter in $\H$, then there exists some $B \in \H$ such that for all $n$, $B/n \leq A_{r_n(B)}$. 
    \end{enumerate}
\end{lemma}

\begin{proof}
    \underline{(2)$\implies$(1):} Let $(A_n)_{n<\omega}$ be a $\leq$-decreasing sequence in $\H$ below some $A \in \H$. For each $a \in E^{[<\infty]}\restrictedto A$, we let $A_a := A_{d_A(a)}$, and if $a \notin E^{[<\infty]}\restrictedto A$ then we let $A_a := A$. By (2), there exists some $B \in \H$ such that $B/n \leq A_{r_n(B)} = A_{d_A(r_n(B))}$ for all $n$, so $B$ diagonalises $(A_n)_{n<\omega}$ below $A$. 

    \underline{(1)$\implies$(2):} Suppose that $(A_a)_{a \in E^{[<\infty]}}$ generates a filter in $\H$. Let $A := A_\emptyset$. For each $n$, let:
    \begin{align*}
        \F_n := \{a \in E^{[<\infty]}\restrictedto A : d_A(a) \leq n\}.
    \end{align*}
    Since the underlying field is finite, $\F_n$ is finite for all $n$. Thus, we may inductively define a $\leq$-decreasing sequence $(A_n)_{n<\omega}$ in $\H$ such that $A_n \leq A_a$ for all $a \in \F_n$. Since $\H$ is selective, we may let $B \in \H$ be such that $B/n \leq A_{d_A(r_n(B))} \leq A_{r_n(B)}$ for all $n$. Then $B$ is the desired element in $\H$.
\end{proof}

Note that the assumption that the underlying field is finite is not required in the direction (2)$\implies$(1).

We conclude the section by examining the relationship among the selectivity property, the (p)-property and the (q)-property.

\begin{lemma}
\label{lem:(q)-property.refines.diagonalisation}
    Let $\H \subseteq E^{[\infty]}$ be a (q)-semicoideal. Let $(A_n)_{n<\omega}$ be a $\leq$-decreasing sequence in $\H$. If $B \in \H\restrictedto A_0$ weakly diagonalises $(A_n)_{n<\omega}$, then there exists some $C \in \H\restrictedto B$ which diagonalises $(A_n)_{n<\omega}$ below $A_0$.
\end{lemma}

\begin{proof}
    Let $A := A_0$, and let $B \in \H\restrictedto A$ weakly diagonalise $(A_n)_{n<\omega}$. Let $(I_m)_{m<\omega}$ be any increasing sequence of intervals such that $B/\max(I_m) \leq A_{d_A(r_{\min(I_m)-1}(B))}$ for all $m$. By the (q)-property, there exists some $C \in \H\restrictedto B$ such that for all $n$, there exists some $m$ such that $\c{r_n(C)} \subseteq \c{r_{\min(I_m)-1}(B)}$ and $C/n \leq B/\max(I_m)$. The first property implies that $d_A(r_n(C)) \leq d_A(r_{\min(I_m)}(B))$, so by the second property:
    \begin{align*}
        C/n \leq B/\max(I_m) \leq A_{d_A(r_{\min(I_m)}(B))} \leq A_{d_A(r_n(C))},
    \end{align*}
    so $C$ diagonalises $(A_n)_{n<\omega}$ below $A_0$.
\end{proof}

\begin{proposition}
\label{prop:selective.iff.(p,q)}
    Let $\H \subseteq E^{[\infty]}$ be a semicoideal. The following are equivalent:
    \begin{enumerate}
        \item $\H$ is selective.

        \item $\H$ satisfies both the (p)-property and the (q)-property.
    \end{enumerate}
\end{proposition}

\begin{proof} 
    \underline{(2)$\implies$(1):} Let $(A_n)_{n<\omega}$ be a $\leq$-decreasing sequence in $\H$, and let $A := A_0$. By the (p)-property, there exists some $B \in \H\restrictedto A$ which weakly diagonalises $(A_n)_{n<\omega}$. Since $\H$ satisfies the (q)-property, by Lemma \ref{lem:(q)-property.refines.diagonalisation} there exists some $C \in \H\restrictedto B$ which diagonalises $(A_n)_{n<\omega}$. Therefore, $\H$ is selective.

    \underline{(1)$\implies$(2):} Since every diagonalisation is a weak diagonalisation, $\H$ satisfies the (p)-property. We shall show that $\H$ satisfies the (q)-property. Let $A \in \H$, and let $(I_m)_{m<\omega}$ be any increasing sequence of finite intervals. For each $n$, let $m(n)$ be the least integer such that $\min(I_{m(n)}) > n$. Let $B \in \H\restrictedto A$ diagonalise $(A/\max(I_{m(n)}))_{n<\omega}$ below $A$. Then for any $n$, if $k := d_A(r_n(B))$, then:
    \begin{align*}
        \c{r_n(B)} \subseteq \c{r_k(A)} \subseteq \c{r_{\min(I_{m(k)-1})}(A)},
    \end{align*}
    and:
    \begin{align*}
        B/n \leq A_{d_A(r_n(B))} = A/\max(I_{m(k)}),
    \end{align*}
    as desired.
\end{proof}

We shall examine the selectivity of semicoideals generated by almost disjoint families in the next section.

\section{Almost disjoint families of subspaces}
\label{sec:ad.families}
We recall the following:
\begin{definition}
    Two infinite-dimensional subspaces $V,W \subseteq E$ are \emph{almost disjoint} if $V \cap W$ is a finite-dimensional subspace of $E$. Otherwise, we say that $V$ and $W$ are \emph{compatible}.
\end{definition}

\begin{definition}    
    Let $\A$ be a family of infinite-dimensional subspaces of $E$. We say that $\A$ is \emph{almost disjoint} if all subspaces in $\A$ are pairwise almost disjoint. We say that $\A$ is \emph{maximal almost disjoint} (or just \emph{mad}) if $\A$ is not strictly contained in another almost disjoint family of infinite-dimensional subspaces.
\end{definition}

Since every infinite-dimensional subspace of $E$ contains a block subspace, an almost disjoint family of subspaces $\A$ is mad iff there is no block subspace which is almost disjoint from every element of $\A$.

We dedicate this section to some basic definitions/notations and results of almost disjoint families of block subspaces, and their relationship with the local Ramsey theory of semicoideals of block subspaces.

\subsection{Uncountability of almost disjoint families}
We begin with the most fundamental question of the theory of mad families --- the uncountability of infinite mad families. Consider the following two cardinals:
\begin{align*}
    \avec^* &:= \min\{|\A| : \text{$\A$ is an infinite mad family of subspaces over $\mathbb{F}$}\}, \\
    \avec &:= \min\{|\A| : \text{$\A$ is an infinite mad family of block subspaces over $\mathbb{F}$}\}.
\end{align*}
It is clear that $\avec^* \leq \avec \leq \frak{c}$. By Proposition 2.5 of \cite{S19}, $\aleph_1 \leq \avec$ (in fact, $\non(\M) \leq \avec$), so $\avec$ is a cardinal invariant. One may ask if the same holds for $\avec^*$ --- that is, if $\aleph_1 \leq \avec^*$ is true. Smythe's proof of $\aleph_1 \leq \avec$ heavily relies on the following lemma:

\begin{lemma}[Lemma 2.3, \cite{S19}]
\label{lem:iteration.lemma.for.subspaces}
    Let $A \in E^{[\infty]}$ be an infinite block sequence, and let $a \in E^{[<\infty]}$ be a finite block sequence. Then there exists some $N < \omega$, that depends only on $A$ and $\max(\supp(a))$, such that for all $x > N$ for which $x \notin \c{A}$, 
    \begin{align*}
        \c{a^\frown x} \cap \c{A} = \c{a} \cap \c{A}.
    \end{align*}
\end{lemma}

He remarked that removing the ``block'' requirement from this lemma seems difficult. $\aleph_1 \leq \avec^*$ was also stated as Proposition 3.1 of \cite{K00}, but Kolman further assumed that in the countable almost disjoint family of subspaces $\{V_n : n < \omega\}$, $V_n \cap \sum_{i \in I} V_i$ is finite-dimensional for all finite $n \notin I \subseteq \omega$. This is generally not satisfied by families of subspaces that are pairwise almost disjoint. The author raised this question on MathOverflow, of which user527492 provided an affirmative answer (see \cite{u25}). 

\begin{lemma}
\label{lem:avec*.geq.aleph_1}
     $\aleph_1 \leq \avec^*$.
\end{lemma}

\begin{proof}
    Let $\{V_n : n < \omega\}$ be a countable family of almost disjoint subspaces of $E$, and we wish to find some infinite-dimensional subspace $W$ which is almost disjoint from $V_n$ for all $n$. We shall first prove a useful claim.

    \begin{claim}
    \label{claim:finite-dimensional.addition}
        Let $U,V,W \subseteq E$ be subspaces such that $V$ and $W$ are almost disjoint, and $U$ is finite-dimensional. Then $U + V$ and $W$ are almost disjoint.
    \end{claim}

    \begin{midproof}
        Define a linear map $\phi : V \times W \to E$ by stipulating that $\phi(v,w) := w - v$. Let $\rho : E \to E/U$ be the canonical quotient map, and let $\psi := \rho \circ \phi$. Note that:
        \begin{align*}
            \dim(\ker(\psi)) \leq \dim(\ker(\phi)) + \dim(\ker(\rho)) < \infty.
        \end{align*}    
        Let $\pi : V \times W \to W$ be the projection map. Observe that for all $w \in W$, we have that:
        \begin{align*}
            &\iffbreak w \in (U + V) \cap W \\
            &\iff w = u + v \text{ for some $u \in U$, $v \in V$} \\
            &\iff w - v \in U \text{ for some $v \in V$} \\
            &\iff (v,w) \in \ker(\psi) \text{ for some $v \in V$}.
        \end{align*}
        This implies that $\pi$ maps $\ker(\psi)$ surjectively onto $(U + V) \cap W$, and since $\ker(\psi)$ is of finite dimension, so is $(U + V) \cap W$.
    \end{midproof}

    We shall define the subspace $W = \c{x_n : n < \omega}$ by defining $x_n$ inductively as follows: Let $x_0 \in V_0$ be any non-zero vector. Now assume that we have defined a linearly independent set $\{x_i : i < n\}$ such that for all $k < n$, we have that:
    \begin{align*}
        \c{x_i : i < n} \cap V_k \subseteq \c{x_i : i \leq k}.
    \end{align*}
    We wish to find some vector $x_n \in V_n$, linearly independent from $\{x_i : i < n\}$, such that for all $k < n$, we have that:
    \begin{align*}
        \c{x_i : i \leq n} \cap V_k \subseteq \c{x_i : i \leq k}.
    \end{align*}
    If we can complete this induction step, then we are done as we have that $W \cap V_n \subseteq \c{x_0,\dots,x_n}$ is finite-dimensional for all $n$.
    
    For each $k < n$, let $U_k := V_k + \c{x_i : i < n}$. By Claim \ref{claim:finite-dimensional.addition}, $V_n \cap U_k$ is of finite dimension. Since no infinite-dimensional vector space is a finite union of finite-dimensional subspaces, we have that:
    \begin{align*}
        V_n \setminus \bigcup_{k<n} V_n \cap U_k \neq \emptyset,
    \end{align*}
    so we may let $x_n$ be any vector in this set. We shall show that this vector $x_n$ works. We first note that $\{x_i : i \leq n\}$ is linearly independent, as $\c{x_i : i < n} \subseteq U_0$ and $x_n \notin U_0$. Now let $k < n$ and $x \in \c{x_i : i \leq n} \cap V_k$, so we have that $x = \sum_{i \leq n} \xi_ix_i \in V_k$. If $\xi_n \neq 0$, then we have that:
    \begin{align*}
        x_n = \xi_n^{-1}\bb{x - \sum_{i<n}\xi_i x_i} \in V_n \cap U_k,
    \end{align*}
    contradicting the choice of $x_n$. Therefore, $\xi_n = 0$, which implies that:
    \begin{align*}
        x \in \c{x_i : i < n} \cap V_k = \c{x_i : i \leq k} \cap V_k,
    \end{align*}
    as desired.
\end{proof}

Smythe asked in \cite{S19} whether $\avec^* = \avec$ holds in $\ZFC$, and this problem remains open to the author's knowledge. In the same paper, Smythe showed that every finite almost disjoint family of block subspaces of cardinality greater than $1$ is not maximal. We extend this fact to arbitrary subspaces.

\begin{proposition}
    Let $\A$ be a finite almost disjoint family of subspaces of $E$ such that $|\A| > 1$. Then $\A$ is not maximal.
\end{proposition}

\begin{proof}
    Let $\A = \{V_k : k < n\}$ be a finite almost disjoint family of subspaces of $E$. 

    \begin{claim}
    \label{claim:E.not.union}
        $E \neq \bigcup_{k<n} V_k$.
    \end{claim}

    \begin{midproof}
        We first assume that for all $j \neq k$, $V_j \cap V_k = \{0\}$. Let $\mathbb{F}$ be the field that $E$ is defined over. If $\mathbb{F}$ is infinite, then we are done, as any subspace over an infinite field is not a finite union of proper subspaces. Assume that $|\mathbb{F}| = N < \infty$, and suppose for a contradiction that $E = \bigcup_{k<n} V_k$. Let $M$ be some large number (which we shall fix later), and let $W := \c{e_n : n < M}$. Let $U_k := W \cap V_k$, and by taking $M$ large enough we may assume that $U_k$ is non-trivial for all $k$. Note that $W = \bigcup_{k<n} U_k$.
        
        Observe that $\dim(W) = M$, so if $U,U' \subseteq W$ are two subspaces such that $\dim(U),\dim(U') > \frac{M}{2}$, then $U$ and $U'$ intersect non-trivially. Since $\{U_k : k < n\}$ is a collection of proper subspaces of $W$ with pairwise trivial intersections, and $n > 2$, there is at most one $k$ such that $\dim(U_k) > \frac{M}{2}$. We may assume that if such a $k$ exists, then $\dim(U_0) > \frac{M}{2}$. Note that since $U_k$ is non-trivial for $k > 0$ and $U_0 \cap U_k \neq \emptyset$, we have that $\dim(U_0) \leq M - 1$. Consequently, for each $k > 0$, $U_k$ contains at most $N^{\frac{M}{2}}$ vectors, which implies that:
        \begin{align*}
            \mod{\bigcup_{k<n} U_k} \leq (n -1) \cdot N^{\frac{M}{2}} + N^{M-1}.
        \end{align*}
        On the other hand, $|W| = N^M$. Since:
        \begin{align*}
            N^M - ((n - 1)N^{\frac{M}{2}} + N^{M-1}) &= (N - 1)N^{M-1} - (n - 1)N^{\frac{M}{2}} \\
            &\geq N^{M-1} - (n-1)N^\frac{M}{2} \\
            &= (N^{\frac{M}{2} - 1} - n + 1)N^\frac{M}{2} \\
            &> 0 \text{ for $M$ large enough}
        \end{align*}
        we have that $W \supsetneq \bigcup_{k<n} U_k$, a contradiction.

        For the general case, suppose that $\{V_k : k < n\}$ is almost disjoint and $E = \bigcup_{k<n} V_k$. Let $M$ be large enough so that $M > \max(\supp(x))$ for all $x \in V_j \cap V_k$ and $j \neq k$. Let $W' := \c{e_n : n \geq M}$, and we have that $W' = \bigcup_{k<n} W' \cap V_k$. Then $\{W' \cap V_k : k < n\}$ is a collection of infinite-dimensional subspaces of $W'$ with pairwise trivial intersections, contradicting the argument above.
    \end{midproof}

    We shall now construct some $V = \c{x_m : m < \omega} \subseteq E$ which intersects $V_k$ trivially for every $k < n$ by defining $x_m$ inductively. Let $x_0 \in E \setminus \bigcup_{k<n} V_k$ be any vector. Now suppose that we have defined $x_i$ for $i < m$ such that for all $k < n$:
    \begin{align*}
        \c{x_i : i < m} \cap V_k = \{0\}.
    \end{align*}
    For each $k < n$, let $U_k := V_k + \c{x_i : i < m}$. By the claim in the proof of Lemma \ref{lem:avec*.geq.aleph_1}, $\{U_k : k < n\}$ is an almost disjoint family of subspaces of $E$, so by applying Claim \ref{claim:E.not.union} to $\{U_k : k < n\}$, $E \neq \bigcup_{k<n} U_k$. Let $x_m \in E \setminus \bigcup_{k<n} U_k$. Note that $\{x_i : i \leq m\}$ is linearly independent, as $x_i \in U_0$ for all $i < m$ but $x_m \notin U_0$. We shall show that:
    \begin{align*}
        \c{x_i : i \leq m} \cap V_k = \{0\}.
    \end{align*}
    If we can complete this induction step, then we are done as $V \cap V_k = \bigcup_{m<\omega} \c{x_i : i < m} \cap V_k = \{0\}$ for all $k < n$.
    
    Let $y \in \c{x_i : i \leq m} \cap V_k$ for some $k < n$. Then $y = \sum_{i \leq m} \xi_i x_i$ for some $\xi_i \in \mathbb{F}$. If $\xi_m \neq 0$, then we have that:
    \begin{align*}
        x_m = \xi_m^{-1}\bb{y - \sum_{i<m} \xi_i x_i} \in U_k,
    \end{align*}
    contradicting the choice of $x_m$. Therefore, $\xi_m = 0$, which implies that:
    \begin{align*}
        y \in \c{x_i : i < m} \cap V_k = \{0\},
    \end{align*}
    as desired. 
\end{proof}

\subsection{Ideals and semicoideals generated by almost disjoint families}
If $\A$ is an almost disjoint family in $[\omega]^\omega$, then the following families of sets are often studied:
\begin{align*}
    \I(\A) &:= \{Y \subseteq \omega : Y \subseteq^* A_0 \cup \cdots \cup A_{n-1} \text{ for some $A_0,\dots,A_{n-1} \in \A$}\}, \\
    \I^+(\A) &:= \Po(\omega) \setminus \I(\A), \\
    \I^{++}(\A) &:= \{Y \subseteq \omega : \{A \in \A : Y \cap A \text{ is infinite}\} \text{ is infinite}\}, \\
    \H(\A) &:= \I^{++}(\A).
\end{align*}
It is simple to verify that $\I^+(\A) = \I^{++}(\A)$ iff $\A$ is maximal. Motivated by this, Smythe introduced a version of $\H(\A)$ for almost disjoint families of subspaces. $\H(\A)$ is a semicoideal, and Smythe studied the 
conditions that $\A$ must satisfy to make $\H(\A)$ a full semicoideal. 

\begin{definition}
    Let $Y \subseteq E$.
    \begin{enumerate}
        \item $Y$ is \emph{big} if there exists some $A \in E^{[\infty]}$ such that $\c{A} \subseteq Y$.
        
        \item $Y$ is \emph{small} if it is not big.
        
        \item $Y$ is \emph{very small} if for all $Z \subseteq E$ such that $Z$ is small, $Y \cup Z$ is small.
    \end{enumerate}
\end{definition}

Note that if $V \subseteq E$ is a subspace, then $V$ is big iff $V$ is infinite-dimensional.

If $E$ is a vector space over $\mathbb{F}_2$, then by Hindman's theorem, every small set is very small. On the other hand, if $E$ is a vector space over $\mathbb{F} \neq \mathbb{F}_2$, then this is no longer true --- fix any $\mu \in \mathbb{F} \setminus \{0,1\}$. We may define the set:
\begin{align*}
    Y := \{x \in E : x = e_n + y \text{ for some $e_n < y$ or $y = 0$}\}.
\end{align*}
Observe that $Y$ and $Y^c$ are both small --- for any $A \in E^{[\infty]}$, let $x \in \c{A}$ be any non-zero vector. Then $x = \lambda_n(x)e_n + y$ for some $e_n < y$ or $y = 0$ and $\lambda_n(x) \neq 0$, so $\frac{1}{\lambda_n(x)}x \in \c{A} \cap Y$ and $\frac{\mu}{\lambda_n(x)}x \in \c{A} \cap Y^c$. However, since $E = Y \cup Y^c$ and $E$ is big, $Y$ is not very small. 

\begin{lemma}
\label{lem:ad.iff.tail.disjoint}
    Let $A \in E^{[\infty]}$ and let $V \subseteq E$ be a subspace. Then $\c{A} \cap V$ is small iff $\c{A/N} \cap V = \emptyset$ for some $N$. Consequently:
    \begin{enumerate}
        \item If $A,B \in E^{[\infty]}$, then $A$ and $B$ are almost disjoint iff $\c{A/N} \cap \c{B} = \emptyset$ for some $N$.

        \item If $A/N$ and $B$ are almost disjoint for some $N$, then $A$ and $B$ are almost disjoint.
    \end{enumerate}
\end{lemma}

\begin{proof}
    Suppose that $\c{A} \cap V$ is big. Let $B \in E^{[\infty]}$ such that $\c{B} \subseteq \c{A} \cap V$. For any $N$, there exists some $M$ such that $B/M \leq A/N$. This implies that $\c{B/M} \subseteq \c{A/N} \cap V$, so in particular, $\c{A/N} \cap V \neq \emptyset$ for all $N$.

    Conversely, suppose that $\c{A/N} \cap V \neq \emptyset$ for all $N$. We construct some $B = (y_n)_{n<\omega} \in E^{[\infty]}$ such that $\c{B} \subseteq \c{A} \cap V$ inductively as follows: Assuming that $y_0,\dots,y_{n-1}$ have been defined so far, let $N$ be large enough so that $\supp(y_{n-1}) < A/N$. We let $y_n \in \c{A/N} \cap V$ be any element. Since $\c{A} \cap V$ is a subspace, we have that $\c{B} \subseteq \c{A} \cap V$, so $\c{A} \cap V$ is big.
\end{proof}

\begin{lemma}
\label{lem:ad.is.very.small}
    Let $U,V \subseteq E$ be subspaces.
    \begin{enumerate}
        \item If $U \cap V$ is small, then $U \cap V$ is very small.

        \item If $U \cap V$ is small, then $U \setminus V$ is very small.
    \end{enumerate}
    Consequently, if $U$ is infinite-dimensional, then it is not possible to have both $U \cap V$ and $U \setminus V$ to be small.
\end{lemma}

\begin{proof}
    Assume that $U \cap V$ is small. Let $Y \subseteq E$, and suppose that $\c{A} \subseteq (U \cap V) \cup Y$ for some $A = (x_n)_{n<\omega} \in E^{[\infty]}$. Since $U \cap V$ is a finite-dimensional space, we may let $N$ be large enough so that $\supp(x_N) > y$ for all $y \in U \cap V$. Then $\c{A/N} \cap (U \cap V) = \emptyset$, so $\c{A/N} \subseteq Y$. Therefore, $Y$ is big.

    Now assume that $U \setminus V$ is small. Let $Y \subseteq E$, and suppose that $\c{A} \subseteq (U \setminus V) \cup Y$. We consider three cases. For the first case, suppose that $\c{A} \cap U$ is small. By Lemma \ref{lem:ad.iff.tail.disjoint}, let $N$ be large enough so that $\c{A/N} \cap U = \emptyset$. Then this implies that $\c{A/N} \subseteq Y$, so $Y$ is big.

    For the second case, suppose that $\c{A} \cap V$ is big. Let $B \in E^{[\infty]}$ be such that $\c{B} \subseteq \c{A} \cap V$. Then $\c{B} \subseteq (U \setminus V) \cup Y$, and since $\c{B} \cap (U \setminus V) = \emptyset$, we have that $\c{B} \subseteq Y$. Therefore, $Y$ is big.

    For the last case, suppose that $\c{A} \cap U$ is big but $\c{A} \cap V$ is small. Let $B \in E^{[\infty]}$ be such that $\c{B} \subseteq \c{A} \cap U$, and note that $\c{B} \cap V$ is also small. By Lemma \ref{lem:ad.iff.tail.disjoint}, let $N$ be large enough so that $\c{B/N} \cap V = \emptyset$. Since $\c{B/N} \subseteq U$, this implies that $\c{B/N} \subseteq U \setminus V$, contradicting our assumption that $U \setminus V$ is small.
\end{proof}

\begin{definition}
\label{def:ad.semi-ideals.and.semicoideals}
    Let $\A$ be an almost disjoint family of subspaces. We then define the following:
    \begin{align*}
        \I(\A) &:= \ss{Y \subseteq E : \m{Y \setminus (V_0 \cup \cdots \cup V_{n-1}) \text{ is small} \\ \text{for some } V_0,\dots,V_{n-1} \in \A}}, \\
        \I^+(\A) &:= \Po(E) \setminus \I(\A), \\
        \I^{++}(\A) &:= \{Y \subseteq E : \{V \in \A : Y \cap V \text{ is big}\} \text{ is infinite}\}, \\
        \H^-(\A) &:= \{B \in E^{[\infty]} : \c{B} \in \I^+(\A)\}, \\
        \H(\A) &:= \{B \in E^{[\infty]} : \c{B} \in \I^{++}(\A)\}.
    \end{align*}
\end{definition}

\begin{lemma}
\label{lem:plusplus.in.plus}
    Let $\A$ be an infinite almost disjoint family of subspaces. Then $\I^{++}(\A) \subseteq \I^+(\A)$ (so $\H(\A) \subseteq \H^-(\A)$).
\end{lemma}

\begin{proof}
    Let $Y \in \I(\A)$, and let $V_0,\dots,V_{n-1} \in \A$ be such that $Y \setminus (V_0 \cup \cdots \cup V_{n-1})$ is small. Then for any $U \in \A$ such that $U \neq V_i$ for any $i < n$, we have that:
    \begin{align*}
        Y \cap U &\subseteq \bb{Y \setminus \bigcup_{i<n} V_i} \cup \bb{U \cap \bigcup_{i<n} V_i} \\
        &= \bb{Y \setminus \bigcup_{i<n} V_i} \cup \bigcup_{i<n} V_i \cap U.
    \end{align*}
    By Lemma \ref{lem:ad.is.very.small}, $V_i \cap U$ is very small for all $i < n$, so $\bigcup_{i<n} V_i \cap U$ is very small. Since $Y \setminus \bigcup_{i<n} V_i$ is small, $Y \cap U$ is small. Therefore:
    \begin{align*}
        \{U \in \A : Y \cap U \text{ is big}\} \subseteq \{V_0,\dots,V_{n-1}\},
    \end{align*}
    so $Y \notin \I^{++}(\A)$. 
\end{proof}

\begin{lemma}
\label{lem:I^+(A).dichotomy}
    Let $\A$ be an infinite almost disjoint family of subspaces, and let $U \subseteq E$ be a subspace. The following are equivalent:
    \begin{enumerate}
        \item $U \in \I^+(\A)$.

        \item $U \in \I^{++}(\A)$ or there exists some $B \in E^{[\infty]}$, with $\c{B} \subseteq U$, such that $B$ is almost disjoint from every element of $\A$.
    \end{enumerate}
    Consequently, if $\A$ is mad, then $\I^+(\A) = \I^{++}(\A)$ (hence $\H^-(\A) = \H(\A)$).
\end{lemma}

\begin{proof}    
    \underline{(1)$\implies$(2):} Let $U \in \I^+(\A)$. Suppose that for all $B \in E^{[\infty]}$ such that $\c{B} \subseteq U$, $B$ is not almost disjoint from some element of $\A$. We let $V_0 \in \A$ be compatible with $U$ (i.e. $U \cap V_0$ is big), and let $B_0 \in E^{[\infty]}$ be such that $\c{B_0} \subseteq U \cap V_0$. Clearly $\c{B_0} \in \I(\A)$, so we have that $U \setminus \c{B_0} \in \I^+(\A)$ (for otherwise $U \in \I(\A)$). Let $B_1' \in E^{[\infty]}$ be such that $\c{B_1'} \subseteq U \setminus \c{B_0}$. By our assumption, $B_1'$ is compatible with some $V_1 \in \A$, so let $B_1 \in E^{[\infty]}$ be such that $\c{B_1} \subseteq \c{B_1'} \cap V_1$. By a similar reasoning as above, we have that $U \setminus (\c{B_0} \cup \c{B_1}) \in \I^+(\A)$. We may repeat this process indefinitely, and we have that $\{V_n : n < \omega\} \subseteq \A$ is such that $U$ is compatible with $V_n$ for all $n$. Therefore, $U \in \I^{++}(\A)$.

    \underline{(2)$\implies$(1):} If $U \in \I^{++}(\A) \subseteq \I^+(\A)$, then we are done, so assume that there exists some $B \in E^{[\infty]}$ such that $\c{B} \subseteq U$ and is almost disjoint from every element of $\A$. Then for any $V_0,\dots,V_{n-1} \in \A$, by Lemma \ref{lem:ad.iff.tail.disjoint} we may let $N$ be large enough so that $\c{B/N} \cap V_i = \emptyset$ for all $i < n$. This implies that $\c{B/N} \subseteq U \setminus \bigcup_{i<n} V_i$. Therefore, $B/N \in \H^-(\A)$, so $U \in \I^+(\A)$ as $\I^+(\A)$ is $\leq$-upward closed.
\end{proof}

We conclude this part with three useful lemmas that we shall use repeatedly in this paper.

\begin{lemma}
\label{lem:infinite.compatibility.lemma.1}
    Let $A \in E^{[\infty]}$, and suppose that $\{B_n : n < \omega\} \subseteq E^{[\infty]}$ is such that $B_n \leq A$ for all $n$. Then there exists some $B \leq A$ such that $B$ is compatible with $B_n$ for all $n$, and:
    \begin{align*}
        \c{B} \subseteq \bigc{\bigcup_{n<\omega} \c{B_n}} \subseteq \c{A}. \tag{$*$}
    \end{align*}
\end{lemma}

\begin{proof}
    Let $\Gamma : \omega \to \omega$ be any infinite-to-one surjection, i.e. a surjection such that $\Gamma^{-1}(m)$ is infinite for all $m < \omega$. We define $B = (x_n)_{n<\omega}$ inductively as follows: If $(x_k)_{k<n}$ has been defined, let $x_n \in \c{B_{\Gamma(n)}}$ be any element such that $(x_k)_{k<n} < x_n$. Then $(x_n)_{n\in\Gamma^{-1}(m)}$ witnesses that $B$ and $B_m$ are compatible for all $m < \omega$, and since $x_n \in \bigcup_{m<\omega} \c{B_m}$ for all $n$, the identity ($*$) is satisfied. 
\end{proof}

\begin{lemma}
\label{lem:infinite.compatibility.lemma.2}
    Let $\A$ be an infinite almost disjoint family of subspaces, and let $U \subseteq E$ be a subspace. Suppose further that $U \in \I^+(\A)$ and there does not exist any $B \in E^{[\infty]}$ such that $\c{B} \subseteq U$ and $B$ is almost disjoint from every element of $\A$. Then there exists an infinite set $\{V_n : n < \omega\} \subseteq \A$ and an infinite disjoint family $\{B_n : n < \omega\}$ such that for all $n$, $\c{B_n} \subseteq V_n \cap U$. .
\end{lemma}

\begin{proof}
    This is similar to showing (1)$\implies$(2) of Lemma \ref{lem:I^+(A).dichotomy}. Since $U \in \I^+(\A)$, $U$ is big, so there exists some $B_0'$ such that $\c{B_0'} \subseteq U$. By our assumption, $B_0'$ is compatible with some $V_0 \in \A$. Let $B_0$ be such that $\c{B_0} \subseteq \c{B_0'} \cap V_0$. Since $U \in \I^+(\A)$, $U \setminus V_0$ is big. Given some $n < \omega$, suppose that we have defined an infinite disjoint family $\{B_k : k < n\}$ and $\{V_k : k < n\} \subseteq \A$ such that $\c{B_k} \subseteq V_k \cap U$ for all $k < n$. Note that $U \setminus \bigcup_{k<n} \c{B_k}$ is not small as $U \in \I^+(\A)$. Thus, we may let $B_n'$ be such that $\c{B_n'} \subseteq U \setminus \bigcup_{k<n} \c{B_k}$, and let $V_n \in \A$ be compatible with $B_n'$. We then let $B_n$ be such that $\c{B_n} \subseteq \c{B_n'} \cap V_n$, and $\{B_n : n < \omega\}$ and $\{V_n : n <\omega\}$ are as desired.
\end{proof}

\begin{lemma}
\label{lem:block.refinement.lemma}
    Let $\A$ be an infinite almost disjoint family of subspaces, and let $U \subseteq E$ be a subspace. If $U \in \I^+(\A)$, then there exists some $B \in \H^-(\A)$ such that $\c{B} \subseteq U$. If furthermore $U \in \I^{++}(\A)$, then $B$ may be chosen such that $B \in \H(\A)$.
\end{lemma}

\begin{proof}
    If $U \notin \I^{++}(\A)$, then by Lemma \ref{lem:I^+(A).dichotomy} there exists some $B \in E^{[\infty]}$ such that $\c{B} \subseteq U$ and $B$ is almost disjoint from every element of $\A$. Then $B \in \H^-(\A)$ and $\c{B} \subseteq U$, so we're done. Otherwise, let $\{V_n : n < \omega\} \subseteq \A$ be such that for all $n$, $\c{B_n} \subseteq V_n \cap U$ for some $B_n \in E^{[\infty]}$. By Lemma \ref{lem:infinite.compatibility.lemma.1}, there exists some $B \in E^{[\infty]}$ such that:
    \begin{align*}
        \c{B} \subseteq \left\langle\bigcup_{n<\omega}\c{B_n}\right\rangle \subseteq U,
    \end{align*}
    and $B$ is compatible with $B_n$ for all $n$. This implies that $B$ is compatible with $V_n$ for all $n$, so $B \in \H(\A)$ and $\c{B} \subseteq U$, as desired.
\end{proof}

\subsection{Selectivity properties revisited}
We now return to studying the combinatorial properties of the semicoideals $\H^-(\A)$ and $\H(\A)$. In Mathias' process of proving that there are no analytic mad families of $[\N]^\infty$, he proved that if $\A$ is an almost disjoint family in $[\N]^\infty$, then $\H^-(\A)$ is a selective coideal. In \cite{S19}, Smythe proved that if $\A$ is a mad family of subspaces, then $\H(\A)$ is a (p)-semicoideal. We prove a few extensions of this result. 

\begin{proposition}
\label{prop:H^-(A).is.(p)}
    Let $\A$ be an infinite almost disjoint family of subspaces. Then $\H^-(\A)$ is a (p)-semicoideal.
\end{proposition}

\begin{proof}
    Let $(B_n)_{n<\omega}$ be a $\leq$-decreasing sequence in $\H^-(\A)$. If there exists some $C$ which weakly diagonalises $(B_n)_{n<\omega}$ and $C$ is almost disjoint from every element of $\A$, then $C \in \H^-(\A)$ and we are done, so assume otherwise. Let $C_0$ be a weak diagonalisation of $(B_n)_{n<\omega}$. By our assumption, $C_0$ is compatible with some $V_0 \in \A$, so by taking a further block subspace of $C_0$ if necessary, we may assume that $\c{C_0} \subseteq V_0$. Now assume, for the induction hypothesis, that we have defined $C_k,V_k$ for $k < m$ such that:
    \begin{enumerate}
        \item $C_k \leq B_k$ is a weak diagonalisation of $(B_n)_{n<\omega}$.

        \item $\c{C_k} \subseteq V_k$ and $V_k \in \A$.

        \item If $i,j < m$ then $V_i \neq V_j$.
    \end{enumerate}

    \begin{claim}
        There exists a $\leq$-decreasing sequence $(B_n^m)_{n<\omega}$ such that for all $n$:
        \begin{enumerate}
            \item $B_n^m \leq B_n$, and;

            \item $B_n^m$ is almost disjoint from $V_k$ for all $k < m$.
        \end{enumerate}
    \end{claim}

    \begin{midproof}
        For any $n$, since $B_n \in \H^-(\A)$ there exists some $D_n \leq B_n$ that is disjoint from $V_k$ for all $k < m$. Let $\Gamma : \omega \to \omega$ be any infinite-to-one surjection. We define $D = (x_n)_{n<\omega}$ as follows: Since $D_{\Gamma(0)}$ is almost disjoint from $V_k$ for all $k < m$, we may let $x_0 \in \c{D_{\Gamma(0)}} \setminus \bigcup_{k<m} V_k$ be any vector (which exists by Lemma \ref{lem:ad.is.very.small}). This implies that $\c{x_0} \cap V_k = \emptyset$ for all $k < m$. Now assume that $(x_i)_{i<n}$ has been defined, and for all $k < m$:
        \begin{align*}
            \c{x_0,\dots,x_{n-1}} \cap V_k = \emptyset.
        \end{align*}
        By Claim \ref{claim:finite-dimensional.addition} in the proof of Lemma \ref{lem:avec*.geq.aleph_1}, for each $k < m$ since $D_{\Gamma(n)}$ and $V_k$ are almost disjoint, we have that $D_{\Gamma(n)}$ and $V_k + \c{x_0,\dots,x_{n-1}}$ are also almost disjoint. Let $x_n \in \c{D_{\Gamma(n)}} \setminus \bigcup_{k<m} (V_k + \c{x_0,\dots,x_{n-1}})$ be any vector (again, which exists by Lemma \ref{lem:ad.is.very.small}). We claim that for all $k < m$:
        \begin{align*}
            \c{x_0,\dots,x_n} \cap V_k = \emptyset.
        \end{align*}
        Otherwise, for any $k < m$ let $y \in \c{x_0,\dots,x_n} \cap V_k$ be any vector, so $y = \sum_{i \leq n} \xi_i x_i$. Since $\c{x_0,\dots,x_{n-1}} \cap V_k = \emptyset$, we have that $\xi_n \neq 0$. But this implies that:
        \begin{align*}
            x_n = \frac{1}{\xi_n}\bb{y - \sum_{i<n} \xi_ix_i} \in V_k + \c{x_0,\dots,x_{n-1}},
        \end{align*}
        contradicting our choice of $x_n$. Then $D = (x_n)_{n<\omega}$ is disjoint from $V_k$ for all $k < m$, as:
        \begin{align*}
            \c{D} \cap V_k &= \bigcup_{i<\omega} \c{x_0,\dots,x_i} \cap V_k = \emptyset.
        \end{align*}
        Let $B_n^m := (x_i)_{\Gamma(i) \geq n}$. Since $\Gamma$ is infinite-to-one, $B_n^m$ is an infinite block sequence for all $n,m$. Then $B_n^m \leq B_n$, $(B_n^m)_{n<\omega}$ is a $\leq$-decreasing sequence, and $B_n^m \leq D$, so $B_n^m$ is almost disjoint from $V_k$ for all $k < m$. 
    \end{midproof}
    
    Let $C_m \leq B_m^m$ be a weak diagonalisation of $(B_n^m)_{n<\omega}$. By our assumption at the beginning, there exists some $V_m \in \A$ which is compatible with $C_m$, and we may assume that $\c{C_m} \subseteq V_m$. Since $C_m \leq B_m^m$, and $B_m^m$ is almost disjoint from $V_k$ for all $k < m$, we have that $V_k \neq V_m$ for all $k < m$. This completes the inductive definition of $C_m,V_m$ for $m < \omega$.
    
    Let $\Gamma : \omega \to \omega$ be an infinite-to-one surjection, and define $C := (y_i)_{i<\omega}$ as follows: Let $y_0 \in \c{C_{\Gamma(0)}}$, and let $y_i \in \c{C_{\Gamma(i)}/i}$ be any vector such that $y_{i-1} < y_i$. Observe that for all $m$, we have that:
    \begin{align*}
        \c{y_i : \Gamma(i) = m} \subseteq \c{C} \cap \c{C_m} \subseteq \c{C} \cap V_m,
    \end{align*}
    so $C$ is compatible with $V_m$ for all $m$. Thus, $C \in \H^-(\A)$. It remains to show that $C$ weakly diagonalises $(B_n)_{n<\omega}$. Fix any $n$, and let $N$ be large enough so that $C_k/N \leq B_n$ for all $k < n$. For all $i \geq N$, we have that $y_i \in \c{C_{\Gamma(i)}/i} \subseteq \c{C_{\Gamma(i)}/N}$. If $\Gamma(i) < n$, then $y_i \in \c{C_{\Gamma(i)}/N} \subseteq \c{B_n}$ by our choice of $N$. If $\Gamma(i) \geq n$, then $C_{\Gamma(i)} \leq B_{\Gamma(i)}^{\Gamma(i)} \leq B_{\Gamma(i)} \leq B_n$, so $y_i \in \c{B_n}$. Therefore, $C/N \leq B_n$.
\end{proof}

\begin{proposition}
\label{prop:H^-(A).is.selective}
    Let $\A$ be an infinite almost disjoint family of subspaces. Then $\H^-(\A)$ is a selective semicoideal.
\end{proposition}

\begin{proof}
    By Proposition \ref{prop:H^-(A).is.(p)}, it remains to show that $\H^-(\A)$ is a (q)-semicoideal. Let $A \in \H^-(\A)$, and let $(I_m)_{m<\omega}$ be an increasing sequence of finite intervals. We split our proof into two cases.

    If there is some $B \leq A$ which is almost disjoint from every element of $\A$, then $C \in \H^-(\A)$ for all $C \leq B$, so any ``sufficiently sparse'' subsequence of $B$ would do the trick. More precisely, if $B = (y_n)_{n<\omega}$, then we define a subsequence of integers $(n_k)_{k<\omega}$ as follows: Let $n_0 := 0$, and if $n_k$ is defined, let $m$ be the least integer such that $n_k < \min(I_m)$. We then define $n_{k+1} := \max(I_m) + 1$, and $(y_{n_k})_{k<\omega} \in \H^-(\A)$ satisfies the conclusion of the (q)-property for $A$ and $(I_m)_{m<\omega}$.

    Suppose there is no $B \leq A$ which is almost disjoint from every element of $\A$. By Lemma \ref{lem:infinite.compatibility.lemma.2}, there exists an infinite $\{V_n : n < \omega\} \subseteq \A$ and an infinite disjoint family $\{B_n : n < \omega\}$ such that $\c{B_n} \subseteq \c{A} \cap V_n$ for all $n$. Let $\Gamma : \omega \to \omega$ be any infinite-to-one surjection. We write $A = (x_n)_{n<\omega}$, and define $B = (y_n)_{n<\omega}$ below $A$ as follows: If $(y_k)_{k<n}$ has been defined, we let $m$ be the least integer such that $\c{y_k : k < n} \subseteq \c{x_0,\dots,x_{\min(I_m)-1}}$. We let $y_n \in \c{B_{\Gamma(n)}/\max(I_m)}$ be any element. Then $(y_n)_{n \in \Gamma^{-1}(m)}$ witnesses that $B$ and $B_m$ are compatible for all $m < \omega$, and since $y_n \in \c{A}$ for all $n$, $B \leq A$ and $B \in \H^-(\A)$.
\end{proof}

\section{Constructions of mad families of subspaces}
\label{sec:full.mad.families}
\begin{definition}
\label{def:full.mad.family}
    A mad family $\A$ of subspaces is \emph{full} if $\H(\A)$ is full (see Definitions \ref{def:full.semicoideal} and \ref{def:ad.semi-ideals.and.semicoideals}).
\end{definition}

We dedicate this section to the proof of Theorem \ref{thm:main.full.mad.family}.

\subsection{Block-splitting families}
We introduce/recall the cardinal invariants which are relevant to our proof of Theorem \ref{thm:main.full.mad.family}. If $f,g \in \omega^\omega$, we write $f \leq^* g$ if $f(n) \leq g(n)$ for almost all $n < \omega$. We say that $\F \subseteq \omega^\omega$ is \emph{bounded} if there exists some $g \in \omega^\omega$ such that $f \leq^* g$ for all $f \in \F$, and \emph{unbounded} otherwise. The \emph{bounding number} $\frak{b}$ is defined as:
\begin{align*}
    \frak{b} := \min\{|\F| : \F \subseteq \omega^\omega \text{ is unbounded}\}.
\end{align*}
If $S \in [\omega]^\omega$ and $A \in [\omega]^\omega$, we say that $S$ \emph{splits} $A$ if $A \cap S$ and $A \setminus S$ are both infinite. A \emph{splitting family} is a family $\S \subseteq [\omega]^\omega$ such that for all $A \in [\omega]^\omega$, there exists some $S \in \S$ which splits $A$. The \emph{splitting number} is defined as:
\begin{align*}
    \frak{s} := \min\{|\S| : \S \subseteq [\omega]^\omega \text{ is a splitting family}\}.
\end{align*}
It turns out that both of these cardinals are closely related to \emph{block-splitting families}. If $\Po = \{P_n : n < \omega\}$ is an interval partition, we say that a set $S \in [\omega]^\omega$ \emph{splits} $\Po$ if $\{n < \omega : P_n \subseteq S\}$ and $\{n < \omega : P_n \cap S = \emptyset\}$ are both infinite. A \emph{block-splitting family} is a family $\S \subseteq [\omega]^\omega$ such that every interval partition is split by some $S \in \S$. The following lemma is due to Kamburelis-Weglorz.

\begin{lemma}[\cite{KW96}]
    $\max\{\frak{b},\frak{s}\} = \min\{|\S| : \S \text{ is a block-splitting family}\}$.
\end{lemma}

We abbreviate $\frak{bs} = \max\{\frak{b},\frak{s}\}$. By Corollary 2.11 of \cite{S19}, the identity $\non(\M) \leq \avec$ holds for any countable field $\mathbb{F}$. Since $\frak{b} \leq \non(\M)$ and $\frak{s} \leq \non(\M)$ (see, for instance, \cite{blasshandbook}), we may conclude that:

\begin{lemma}
\label{lem:bs.leq.avec}
    $\frak{bs} \leq \avec$.
\end{lemma}

\subsection{Proof of Theorem \ref{thm:main.full.mad.family}}
\begin{lemma}
\label{lem:splitting.lemma}
    There exists a collection $\{Z_\alpha : \alpha < \frak{bs}\}$ of block subspaces of $E$ such that for all almost disjoint families $\A$ of subspaces and $B \in \H^-(\A)$, the following properties hold:
    \begin{enumerate}[label=(\roman*)]
        \item For all $\alpha < \frak{bs}$, $\c{B} \cap Z_\alpha \in \I^+(\A)$ or $\c{B} \cap Z_\alpha^c \in \I^+(\A)$.
    
        \item If $\alpha < \frak{bs}$ and $\c{B} \cap Z_\alpha \in \I^+(\A)$, then there exists some $C_\alpha^0 \in \H^-(\A)$ such that $\c{C_\alpha^0} \subseteq \c{B} \cap Z_\alpha$.

        \item If $\alpha < \frak{bs}$ and $\c{B} \cap Z_\alpha^c \in \I^+(\A)$, then there exists some $C_\alpha^1 \in \H^-(\A)$ such that $\c{C_\alpha^1} \subseteq \c{B} \cap Z_\alpha^c$.

        \item There exists some $\alpha < \frak{bs}$ such that $\c{B} \cap Z_\alpha \in \I^+(\A)$ and $\c{B} \cap Z_\alpha^c \in \I^+(\A)$.
    \end{enumerate}
\end{lemma}

\begin{proof}
    Let $\{S_\alpha : \alpha < \frak{bs}\} \subseteq [\omega]^\omega$ be a block-splitting family, and let $Z_\alpha := \c{(e_n)_{n \in S_\alpha}}$. We remark that $\c{(e_n)_{n \notin S_\alpha}} \subseteq Z_\alpha^c$. We shall show that $\{Z_\alpha : \alpha < \frak{bs}\}$ satisfies the four properties above. 

    \underline{Property (i):} Suppose that $\c{B} \cap Z_\alpha \in \I(\A)$ and $\c{B} \cap Z_\alpha^c \in \I(\A)$. Let $\{V_i : i < n\} \subseteq \A$ such that $(\c{B} \cap Z_\alpha) \setminus \bigcup_{i<n} V_i$ and $(\c{B} \cap Z_\alpha^c) \setminus \bigcup_{i<n} V_i$ are small. Since $B \in \H^-(\A)$, there exists some $C \leq B$ such that:
    \begin{align*}
        \c{C} \subseteq \c{B} \setminus \bigcup_{i<n} V_i.
    \end{align*}
    However, since:
    \begin{align*}
        \c{C} \cap Z_\alpha \subseteq (\c{B} \cap Z_\alpha) \setminus \bigcup_{i<n} V_i, \\        
        \c{C} \cap Z_\alpha^c \subseteq (\c{B} \cap Z_\alpha^c) \setminus \bigcup_{i<n} V_i,
    \end{align*}
    we have that $\c{C} \cap Z_\alpha$ and $\c{C} \cap Z_\alpha^c$ are small. This contradicts Lemma \ref{lem:ad.is.very.small}.

    \underline{Property (ii):} This follows from Lemma \ref{lem:block.subspaces.intersection.lemma}, where we may let $C_\alpha^0 \in E^{[\infty]}$ be such that $\c{C_\alpha^0} = \c{B} \cap \c{(e_n)_{n \in S_\alpha}}$.
    
    \underline{Property (iii):} Suppose that $\c{B} \cap Z_\alpha^c \in \I^+(\A)$. If there exists some $C \in E^{[\infty]}$ such that $\c{C} \subseteq \c{B} \cap Z_\alpha^c$ which is almost disjoint from every element of $\A$, then $C \in \H^-(\A)$ and we are done, so assume otherwise. In a way similar to the one presented in the proof of Lemma \ref{lem:infinite.compatibility.lemma.2}, we define $V_n,B_n$ for $n < \omega$ such that $\bigcup_{n<\omega} \c{B_n} \subseteq \c{B} \cap Z_\alpha^c$, $\{B_n : n < \omega\}$ is a disjoint family of block subspaces, and $V_n \in \A$ is compatible with $B_n$ for all $n$. Let $C_n$ be such that $\c{C_n} \subseteq V_n \cap \c{B_n}$. 
    
    For each $n$, we write $C_n := (x_{n,m})_{m<\omega}$. Since $\c{C_n} \subseteq Z_\alpha^c$, this means that for all $m$, $\lambda_{l_{n,m}}(x_{n,m}) \neq 0$ for some $l_{m,n} \notin S_\alpha$. Now let $\Gamma : \omega \to \omega \times \omega$ be any function such that:
    \begin{enumerate}
        \item for all $n$, there are infinitely many $k$ such that $\Gamma(k) = (n,m)$ for some $m$, and;

        \item if $k < k'$, $\Gamma(k) = (n,m)$ and $\Gamma(k') = (n',m')$, then $x_{n,m} < x_{n',m'}$.
    \end{enumerate}
    This implies that $C := (x_{\Gamma(k)})_{k<\omega}$ is a well-defined infinite block sequences, where if $\Gamma(k) = (n,m)$ then $x_{\Gamma(k)} := x_{n,m}$. Then $C$ is compatible with every $C_n$, so $C \in \H^-(\A)$. Furthermore, we have that:
    \begin{align*}
        \c{C} \subseteq \bigc{\bigcup_{n<\omega} \c{C_n}} \subseteq \c{B},
    \end{align*}
    so $C \leq B$. Finally, if $x \in \c{C}$, then $\lambda_l(x) \neq 0$ for some $l \notin S_\alpha$, so $x \notin Z_\alpha$. In other words, $\c{C} \subseteq Z_\alpha^c$. Therefore, $C \in \H^-(\A)$ and $\c{C} \subseteq \c{B} \cap Z_\alpha^c$, as desired.

    \underline{Property (iv):} We first suppose that $B$ is almost disjoint from every element of $\A$. If $B = (x_n)_{n<\omega}$, we define, for $n < \omega$:
    \begin{align*}
        P_0 &:= \{0,\dots,\max(\supp(x_0))\}, \\
        P_{n+1} &:= \{\max(\supp(x_n))+1,\dots,\max(\supp(x_{n+1}))\}.
    \end{align*}
    We have that $\Po := \{P_n : n < \omega\}$ is an interval partition such that $\supp(x_n) \subseteq P_n$ for all $n$. Since $\{S_\alpha : \alpha < \frak{bs}\}$ is a block-splitting family, there exists some $\alpha < \frak{bs}$ that splits $\Po$. Let $T_0 := \{n < \omega : P_n \subseteq S_\alpha\}$ and $T_1 := \{n < \omega : P_n \cap S_\alpha = \emptyset\}$. Define $B_0 := (x_n)_{n \in T_0}$ and $B_1 := (x_n)_{n \in T_1}$. Since $B_0,B_1 \leq B$, and $B$ is almost disjoint from every element of $\A$, so are $B_0,B_1$ and so $B_0,B_1 \in \H^-(\A)$. If $n \in T_0$, then $\supp(x_n) \subseteq P_n \subseteq S_\alpha$, $x_n \in (e_n)_{n \in S_\alpha}$. Therefore, $\c{B_0} \subseteq \c{B} \cap Z_\alpha$. By a similar reasoning, we may also conclude that $\c{B_1} \subseteq \c{B} \cap Z_\alpha^c$.

    Now suppose that there is no $C \leq B$ which is almost disjoint from every element of $\A$ (if such a $C$ exists, we may apply the argument in the previous paragraph to $C$, and use the fact that $\I^+(\A)$ is $\subseteq$-upward closed). By Lemma \ref{lem:infinite.compatibility.lemma.2}, let $\{V_n : n < \omega\} \subseteq \A$ be an infinite set and let $\{B_n : n < \omega\}$ be an infinite disjoint family such that $\c{B_n} \subseteq \c{B} \cap V_n$ for all $n$. Let $\Gamma : \omega \to \omega$ be an infinite-to-one surjection. We define an infinite block sequence $(x_n)_{n<\omega}$ inductively as follows: If $(x_k)_{k<n}$ has been defined, let $x_n \in \c{B_{\Gamma(n)}}$ be any vector $x_{n-1} < x_n$. Having defined $(x_n)_{n<\omega}$, we now define an interval partition $\Po := \{P_m : m < \omega\}$ as follows: Let $P_0 := \{0,\dots,\max(\supp(x_{n_0}))\}$, where $n_0$ is the least integer such that $\Gamma(n_0) = 0$. If $P_m$ has been defined with $\max(P_m) = n_m$, Let $\min(P_{m+1}) = n_m + 1$, and let $\max(P_{m+1})$ be large enough so that for all $i \leq m+1$, there exists some $n$ such that $\Gamma(n) = i$ and $\supp(x_n) \subseteq P_{m+1}$. This is possible as $\Gamma$ is infinite-to-one. 

    Now let $\alpha < \frak{bs}$ be such that $S_\alpha$ splits $\Po$. For each $i$, let $T_i^0 := \{n < \omega : \Gamma(n) = i \wedge \supp(x_n) \subseteq S_\alpha\}$ and $T_i^1 := \{n < \omega : \Gamma(n) = i \wedge \supp(x_n) \cap S_\alpha = \emptyset\}$. By our construction of $\Po$, we have that $T_i^0$ and $T_i^1$ are infinite for all $i$. By a similar reasoning to the first case, we may conclude that $\c{B} \cap Z_\alpha$ and $\c{B} \cap Z_\alpha^c$ are compatible with $V_i$ for all $i$, so they are both in $\I^+(\A)$.
\end{proof}

We are now ready to prove the theorem.

\begin{proof}[Proof of Theorem \ref{thm:main.full.mad.family}]
    Let $\{Z_\alpha : \alpha < \frak{bs}\}$ be a family of subsets of vectors satisfying the conclusions of Lemma \ref{lem:splitting.lemma}. We write $Z_\alpha^0 := Z_\alpha$ and $Z_\alpha^1 := Z_\alpha^c$ for all $\alpha$. With this, for a fixed almost disjoint family $\A \subseteq E^{[\infty]}$ and $B \in \H^-(\A)$, there exist $\alpha < \frak{bs}$ and $\tau_B^\A \in 2^\alpha$ such that:
    \begin{enumerate}
        \item If $\beta < \alpha$, then $\c{B} \cap Z_\beta^{1-\tau_B^\A(\beta)} \in \I(\A)$ (so $\c{B} \cap Z_\beta^{\tau_B^\A(\beta)} \in \I^+(\A)$ by property (i) of Lemma \ref{lem:splitting.lemma});

        \item $\c{B} \cap Z_\alpha^0 \in \I^+(\A)$ and $\c{B} \cap Z_\alpha^1 \in \I^+(\A)$ (such an $\alpha$ exists by property (iv) of Lemma \ref{lem:splitting.lemma}). 
    \end{enumerate}
    Observe that for a fixed almost disjoint family $\A$, $\tau_B^\A \in 2^{<\frak{bs}}$ is unique, and if $B' \leq B$ and $B' \in \H^-(\A)$, then $\tau_B^\A \subseteq \tau_{B'}^\A$. Let $\Lim(\frak{c})$ be the set of limit ordinals below $\frak{c}$. Let $\{Y_\alpha : \alpha \in \Lim(\frak{c})\}$ be an enumeration of $\Po(E)$, and let $\{B_\alpha : \alpha \in \Lim(\frak{c})\}$ be an enumeration of $E^{[\infty]}$ in a way that for any $B \in E^{[\infty]}$ and $Y \subseteq E$, there are cofinally many $\alpha \in \Lim(\frak{c})$ such that $(Y_\alpha,B_\alpha) = (Y,B)$. We shall define our mad family $\A = \{A_\alpha : \alpha < \frak{c}\}$ recursively, where $\A_\alpha := \{A_\beta : \beta < \alpha\}$, along with a collection $\{A_\alpha' : \alpha \in \Lim(\frak{c})\}$ of block sequences, and an injective sequence $\{\sigma_\alpha : \alpha \in \Lim(\frak{c})\}$ of elements of $2^{<\frak{bs}}$ such that the following holds:
    \begin{enumerate}
        \item If $\beta < \alpha$ are limit ordinals, then $\sigma_\alpha \not\sqsubseteq \sigma_\beta$.

        \item $A_\alpha'$ is almost disjoint with $A_\beta$ for all $\beta < \alpha$.

        \item If $\xi < \dom(\sigma_\alpha)$, then $\c{A_\alpha'} \setminus Z_\xi^{\sigma_\alpha(\xi)}$ is small (hence, by Lemma \ref{lem:ad.is.very.small}, is very small).

        \item $\{A_{\alpha+n} : n < \omega\}$ is a disjoint family below $A_\alpha'$.
    \end{enumerate}
    For our base case, we let $\A_\omega$ be any countable almost disjoint family. Assume that $\alpha \geq \omega$ is a limit ordinal, and that $A_{\beta+n}$, $A_\beta'$ and $\sigma_\beta$ have been defined for limit $\beta < \alpha$ and $n < \omega$, and satisfy the requirements stated above so far. We define the element $B$ depending on the following scenarios:
    \begin{enumerate}[label=(\alph*)]
        \item If $B_\alpha \notin \H(\A_\alpha)$, then let $B := B_\alpha$ if $B_\alpha \in \H^-(\A_\alpha)$. Otherwise, let $B$ be any element of $\H^-(\A_\alpha)$.

        \item Suppose that $B_\alpha \in \H(\A_\alpha)$ (so $B_\alpha \in \H^-(\A_\alpha)$ by Lemma \ref{lem:plusplus.in.plus}), and:
        \begin{align*}
            \cal{C}_\alpha := \{A \in \A_\alpha : B_\alpha \text{ and } A \text{ are compatible}\}.
        \end{align*}
        Suppose further that there exists some $B_\alpha' \leq B_\alpha$ that is almost disjoint from every element of $\cal{C}_\alpha$, and $\c{B_\alpha'} \subseteq Y_\alpha$. Then we let $B := B_\alpha'$.

        \item Suppose that $B_\alpha \in \H(\A_\alpha)$ (so $B_\alpha \in \H^-(\A_\alpha)$), and for all $B \leq B_\alpha$ which is almost disjoint from every element of $\cal{C}_\alpha$, there is no $B' \leq B$ such that $\c{B'} \not\subseteq Y_\alpha$. Then we let $B := B_\alpha$.
    \end{enumerate}
    We inductively define $\{C_s : s \in 2^{<\omega}\} \subseteq \H^-(\A_\alpha)$, $\{\eta_s : s \in 2^{<\omega}\} \subseteq 2^{<\frak{bs}}$ and ordinals $\{\alpha_s : s \in 2^{<\omega}\}$ as follows:
    \begin{enumerate}
        \item $C_\emptyset := B$.

        \item $\eta_s := \tau_{C_s}^{\A_\alpha}$ and $\alpha_s := \dom(\eta_s)$.

        \item $C_{s^\frown0},C_{s^\frown1} \in \H^-(\A_\alpha)$ are such that $\c{C_{s^\frown0}} \subseteq \c{C_s} \cap Z_{\alpha_s}^0$ and $\c{C_{s^\frown1}} \subseteq \c{C_s} \cap Z_{\alpha_s}^1$ (which both exist by properties (ii) and (iii) of Lemma \ref{lem:splitting.lemma}).
    \end{enumerate}
    We then define $\eta_f := \bigcup_{n<\omega} \eta_{f\restrictedto n}$ for any $f \in 2^\omega$, and note that $\eta_f \in 2^{<\frak{bs}}$ as $\frak{bs}$ has uncountable cofinality (as $\frak{b}$ is regular and $\frak{s}$ has uncountable cofinality, see Corollary 21 of \cite{G21}). We note the following observations:
    \begin{enumerate}
        \item If $s \sqsubseteq t$, then $C_s \leq C_t$ and $\eta_s \sqsubseteq \eta_t$.

        \item If $s$ and $t$ are incompatible, then so are $\eta_s$ and $\eta_t$.

        \item Consequently, if $f \neq g$, then $\eta_f$ and $\eta_g$ are incompatible.
    \end{enumerate}
    Since $\alpha < \frak{c}$, we may let $f \in 2^\omega$ be such that $\eta_f \not\sqsubseteq \sigma_\beta$ for all limit ordinals $\beta < \alpha$. We now have a $\leq$-decreasing sequence $(C_{f\restrictedto n})_{n<\omega}$ in $\H^-(\A_\alpha)$, so by Proposition \ref{prop:H^-(A).is.(p)} there exists some $D \in \H^-(\A_\alpha)$ such that $D \leq^* C_{f\restrictedto n}$ for all $n$.

    Let $\sigma_{\alpha} := \eta_f$. Observe that for all $\xi < \dom(\eta_f)$, if $n < \omega$ is such that $\xi \in \dom(\eta_{f\restrictedto n})$, and if $N$ is such that $D/N \leq C_{f\restrictedto n}$, then:
    \begin{align*}
        \c{D/N} \cap Z_\xi^{1-\eta_f(\xi)} = \c{D/N} \cap Z_\xi^{1-\eta_{f\restrictedto n}(\xi)} \subseteq \c{C_{f\restrictedto n}} \cap Z_\xi^{1-\eta_{f\restrictedto n}(\xi)} \in \I(\A_\alpha),
    \end{align*}
    so that for all $\xi < \dom(\sigma_{\alpha})$, $\c{D} \cap Z_\xi^{1-\sigma_\alpha(\xi)} \in \I(\A)$ (as $\c{D} \setminus \c{D/N}$ is very small, by Lemma \ref{lem:ad.is.very.small}). Therefore, for each $\xi < \dom(\sigma_\alpha)$ we may find a finite subset $\F_\xi \subseteq \A_\alpha$ such that $\bb{\c{D} \cap Z_\xi^{1-\sigma_\alpha(\xi)}} \setminus \bigcup_{A \in \F_\xi} \c{A}$ is small. Now let:
    \begin{align*}
        \B_0 &:= \{A_{\beta+n} : \sigma_\beta \sqsubseteq \sigma_\alpha \text{ and } n < \omega\}, \\
        \B_1 &:= \bigcup_{\xi < \dom(\sigma_\alpha)} \F_\xi, \\
        \B &:= \B_0 \cup \B_1.
    \end{align*}
    Note that $\B \subseteq \A_\alpha$, so it is an almost disjoint family.

    \begin{claim}
    \label{claim:B.smaller.than.bs}
        $|\B| < \frak{bs}$.
    \end{claim}

    \begin{midproof}
        It suffices to show that $|\B_0| < \frak{bs}$ and $|\B_1| < \frak{bs}$. Since $\dom(\sigma_\alpha) < \frak{bs}$, there are $<\frak{bs}$ many initial segments of $\sigma_\alpha$. This implies that there are $<\frak{bs}$ many $\beta < \alpha$ such that $\sigma_\beta \sqsubseteq \sigma_\alpha$, as if $\beta,\gamma < \alpha$ and $\beta \neq \gamma$, then $\sigma_\beta \neq \sigma_\gamma$ by the induction hypothesis. Therefore, $|\B_0| < \frak{bs}$. $\B_1$ is a union of $\dom(\sigma_\alpha) < \frak{bs}$ many finite sets, so $|\B_1| < \frak{bs}$ as well. 
    \end{midproof}

    By Lemma \ref{lem:block.subspaces.intersection.lemma},  the set:
    \begin{align*}
        \{\c{B'} \cap \c{D} : B' \in \B \text{ and } \c{B'} \cap \c{D} \text{ is infinite-dimensional}\}
    \end{align*}
    is an almost disjoint family of block subspaces of $\c{D}$. Note that $\c{D}$ is not in the set, as otherwise this implies that $\c{D} \subseteq A_\beta$ for some $\beta < \alpha$, contradicting that $\D \in \H^-(\A_\alpha)$. By Lemma \ref{lem:bs.leq.avec}, $|\B| < \frak{bs} \leq \avec$, so there exists some $A_\alpha' \leq D$ such that $A_\alpha'$ is almost disjoint from every element of $\B$. 

    \begin{claim}
        \hfill
        \begin{enumerate}
            \item For all $\xi \in \dom(\sigma_\alpha)$, $\c{A_\alpha'} \setminus Z_\xi^{\sigma_\alpha(\xi)}$ is small.

            \item For all $\beta < \alpha$, $A_\alpha'$ and $A_\beta$ are almost disjoint.
        \end{enumerate}
    \end{claim}

    \begin{midproof}
        \hfill
        \begin{enumerate}
            \item Note that $\c{A_\alpha'} \setminus Z_\xi^{\sigma_\alpha(\xi)} = \c{A_\alpha'} \cap Z_\xi^{1-\sigma_\alpha(\xi)}$. We have that $\c{A_\alpha'} \cap Z_\xi^{1 - \sigma_\alpha(\xi)} \setminus \bigcup_{A \in \F_\xi} \c{A}$ is small by above, and since $A_\alpha'$ is almost disjoint from every $A \in \F_\xi$, we have that:
            \begin{align*}
                \c{A_\alpha'} \cap Z_\xi^{1 - \sigma_\alpha(\xi)} \cap \bigcup_{A \in \F_\xi} \c{A} &\subseteq \bigcup_{A \in \F_\xi} \c{A_\alpha'} \cap \c{A}.
            \end{align*}
            By Lemma \ref{lem:ad.is.very.small}, $\c{A_\alpha'} \cap \c{A}$ is very small for each $A \in \F_\xi$, so $\c{A_\alpha'} \cap Z_\xi^{1-\sigma_\alpha(\xi)}$ is small. 

            \item Let $\beta + n < \alpha$, where $\beta$ is a limit ordinal. If $A_{\beta+n} \in \B_0$, then $A_\alpha$ is almost disjoint from $A_{\beta+n}$ by construction, so we are done. Now assume that $A_{\beta+n} \notin \B_0$. Let $\xi$ be the least ordinal such that $\sigma_\beta(\xi) \neq \sigma_\alpha(\xi)$. By induction hypothesis, we have that $W := \c{A_\beta'} \setminus Z_\xi^{\sigma_\beta(\xi)} = \c{A_\beta'} \setminus Z_\xi^{1-\sigma_\alpha(\xi)}$ is small (and hence very small). Since $A_{\beta+n} \leq A_\beta'$, we have that for some small set $W'$:
            \begin{align*}
                \c{A_\alpha'} \cap \c{A_{\beta+n}} &\subseteq W \cup (\c{A_\alpha'} \cap Z_\xi^{1-\sigma_\alpha(\xi)}) \\
                &\subseteq W \cup W' \cup \bigcup_{A \in F_\xi} \c{A_\alpha'} \cap \c{A},
            \end{align*}
            which is also small as all the components in the union are very small, as desired.
        \end{enumerate}
    \end{midproof}
    We complete the induction by letting $\{A_{\alpha+n} : n < \omega\}$ be any infinite disjoint family below $A_\alpha'$. We note that $\A$ is a mad family --- otherwise, if $B$ is almost disjoint from every element of $\A$, then $B \in \H^-(\A_\alpha)$ for all $\alpha < \frak{c}$. Let $\alpha$ be any limit ordinal such that $B_\alpha = B$, and we have that $A_\alpha \leq B$, a contradiction. It remains to show that $\A$ is full.
    
    To see that $\A$ is full, let $Y \subseteq E$, and suppose that $Y$ is $\H(\A)$-dense below some $B \in \H(\A)$. Since $\H(\A) = \bigcup_{\alpha<\frak{c}} \H(\A_\alpha)$, we may let $\alpha$ be large enough so that $(Y_\alpha,B_\alpha) = (Y,B)$ and $B \in \H(\A_\alpha)$. This means that in the induction step, we have either case (b) or (c). Observe that (c) is not possible --- since $A_\alpha' \leq B_\alpha$ is almost disjoint from every element of $\cal{C}_\alpha$ as $\cal{C}_\alpha \subseteq \A_\alpha$, and $A_\alpha' \in \H(\A)$ as $A_\alpha'$ is compatible with $A_{\alpha+n}$ for all $n$. However, there is no $B' \leq A_\alpha'$ for which $\c{B'} \subseteq Y_\alpha$, contradicting that $Y$ is $\H(\A)$-dense below $B$. This means that case (b) must have occurred in this induction step. Since $B_{\alpha'}$ is compatible with $A_{\alpha+n}$ for all $n$, $B_{\alpha'} \in \H(\A)$ and $\c{B_\alpha'} \subseteq Y_\alpha$, as desired.
\end{proof}

\subsection{Completely separable mad family of subspaces}
We dedicate this section to defining a notion of completely separable almost disjoint families of subspaces and proving Theorem \ref{thm:cs.mad.family}. 

\begin{definition}
\label{def:cs.mad.family}
    An almost disjoint family $\A$ of subspaces is \emph{completely separable} if for any subspace $U \subseteq E$, if $U \in \I^+(\A)$ then there exists some $V \in \A$ such that $V \subseteq U$.
\end{definition}

\begin{lemma}
    A completely separable almost disjoint family of subspaces is maximal.
\end{lemma}

\begin{proof}
    Suppose that $\A$ is a completely separable almost disjoint family of subspaces that is not maximal. Let $U \subseteq E$ be a subspace almost disjoint from every $V \in \A$. For any $V_0,\dots,V_{n-1} \in \A$, we have that:
    \begin{align*}
        U \setminus \bigcup_{i<n} V_i = U \setminus \bigcup_{i<n} U \cap V_i.
    \end{align*}
    By Lemma \ref{lem:ad.is.very.small}, $U \cap V_i$ is very small for all $i < n$, so $\bigcup_{i<n} U \cap V_i$ is very small. Since $U$ is big, we have that $U \setminus \bigcup_{i<n} V_i$ is big. Since $V_0,\dots,V_{n-1}$ are arbitrary, $U \in \I^+(\A)$. By the complete separability of $\A$, this implies that $V \subseteq U$ for some $V \in \A$, contradicting that $U$ and $V$ are almost disjoint.
\end{proof}

One may consider the alternative definition where we remove the subspace requirement of $U$ in Definition \ref{def:cs.mad.family}. However, it turns out that no such almost disjoint families of subspaces exist.

\begin{proposition}
    Let $\A$ be an infinite almost disjoint family of subspaces. There exists some $Y \in \I^{++}(\A)$ such that for all $A \in \A$, $\c{A} \not\subseteq Y$.
\end{proposition}

\begin{proof}
    Let $\{V_n : n < \omega\} \subseteq \A$. For each $n$, let $B_n,C_n$ be block subspaces of $V_n$ such that $\c{B_n} \cap \c{C_n} = \emptyset$. Let $\Gamma : \omega \to \omega$ be an infinite-to-one surjection. We define a block sequence $(x_m)_{m<\omega}$ as follows: Let $x_0 \in \c{C_{\Gamma(0)}}$ be any vector. If $x_m$ has been defined, let $x_{m+1} \in \c{C_{\Gamma(m+1)}}$ be any vector such that $x_m < x_{m+1}$. We let $D_n := (x_m)_{m \in \Gamma^{-1}(n)}$, which is an infinite block sequence as $\Gamma$ is infinite-to-one. We write $D_n = (y_k^n)_{k<\omega}$ Note that the block subspaces $D_n$ have the properties: 
    \begin{enumerate}
        \item For all $z \in \c{D_n}$, $\supp(z) \subseteq \bigcup_{k<\omega} \supp(y_k^n)$.
        
        \item If $m \neq n$, then $\bigcup_{k<\omega} \supp(y_k^n) \cap \bigcup_{j<\omega} \supp(y_j^m) = \emptyset$.
    \end{enumerate}
    We define:
    \begin{align*}
        Y := \bigcup_{n<\omega} \c{D_n}.
    \end{align*}
    Clearly $Y \in \I^{++}(\A)$, as $\c{D_n} \subseteq V_n \cap Y$ for all $n$. 

    \begin{claim}
        If $B \in E^{[\infty]}$ and $\c{B} \subseteq Y$, then $B \leq D_n$ for some $n$.
    \end{claim}

    \begin{midproof}
        We write $B = (z_k)_{k<\omega}$, and suppose that $z_0 \in \c{D_n}$. We shall show that $z_k \in \c{D_n}$ for all $k$, which implies that $\c{B} \subseteq \c{D_n}$. Suppose that $z_k \in \c{D_m}$ for some $m \neq n$ and $k > 0$. $z_0 + z_k \notin \c{D_n}$, as:
        \begin{align*}
            \supp(z_0 + z_k) = \supp(z_0) \cup \supp(z_k) \not\subseteq \bigcup_{j<\omega} \supp(y_j^n).
        \end{align*}
        Similarly, $z_0 + z_k \notin \c{D_m}$. For each $i \neq m,n$, since:
        \begin{align*}
            \supp(z_0) \cap \bigcup_{j<\omega} \supp(y_j^i) = \emptyset,
        \end{align*}
        we may conclude that $z_0 + z_k \notin \c{D_i}$ either. Therefore, $z_0 + z_k \in \c{B} \setminus Y$, a contradiction.
    \end{midproof}

    We may now finish the proof of the lemma. Suppose that there exists some $V \in \A$ such that $V \subseteq Y$. By the claim above, we have that $V \subseteq \c{D_n} \subseteq V_n$ for some $n$. Since $\A$ is an almost disjoint family, we have that $V = V_n$. But this implies that $\c{B_n} \subseteq V_n \subseteq \c{D_n} \subseteq \c{C_n}$, which contradicts that $\c{B_n} \cap \c{C_n} = \emptyset$. 
\end{proof}

The proof of Theorem \ref{thm:main.full.mad.family} may be adjusted to obtain a completely separable mad family of block subspaces, with the main change being at the induction step. We describe the details here.

\begin{proof}[Proof of Theorem \ref{thm:cs.mad.family}]
    Let $\{Z_\alpha : \alpha < \frak{bs}\}$ be a family of subsets of vectors satisfying the conclusions of Lemma \ref{lem:splitting.lemma}. We write $Z_\alpha^0 := Z_\alpha$ and $Z_\alpha^1 := Z_\alpha^c$ for all $\alpha$. For a fixed almost disjoint family $\A \subseteq E^{[\infty]}$ and $B \in \H^-(\A)$, there exist $\alpha < \frak{bs}$ and $\tau_B^\A \in 2^\alpha$ such that:
    \begin{enumerate}
        \item If $\beta < \alpha$, then $\c{B} \cap Z_\beta^{1-\tau_B^\A(\beta)} \in \I(\A)$ (so $\c{B} \cap Z_\beta^{\tau_B^\A(\beta)} \in \I^+(\A)$ by property (i) of Lemma \ref{lem:splitting.lemma});

        \item $\c{B} \cap Z_\alpha^0 \in \I^+(\A)$ and $\c{B} \cap Z_\alpha^1 \in \I^+(\A)$ (such an $\alpha$ exists by property (iv) of Lemma \ref{lem:splitting.lemma}). 
    \end{enumerate}
    Observe that for a fixed almost disjoint family $\A$, $\tau_B^\A \in 2^{<\frak{bs}}$ is unique, and if $B' \leq B$ and $B' \in \H^-(\A)$, then $\tau_B^\A \subseteq \tau_{B'}^\A$. Let $\{U_\alpha : \alpha < \frak{c}\}$ be an enumeration of all subspaces of $E$. We may need that for any subspace $U \subseteq E$, there are cofinally many $\alpha < \frak{c}$ such that $U_\alpha = U$. We shall define our mad family $\A = \{A_\alpha : \alpha < \frak{c}\}$ recursively, where $\A_\alpha := \{A_\beta : \beta < \alpha\}$, along with an injective sequence $\{\sigma_\alpha : \alpha < \frak{c}\}$ of elements of $2^{<\frak{bs}}$ such that the following holds:
    \begin{enumerate}
        \item If $\beta < \alpha$, then $\sigma_\alpha \not\sqsubseteq \sigma_\beta$.

        \item If $\xi < \dom(\sigma_\alpha)$, then $\c{A_\alpha} \setminus Z_\xi^{\sigma_\alpha(\xi)}$ is small (hence, by Lemma \ref{lem:ad.is.very.small}, is very small).
    \end{enumerate}
    For our base case, we let $\A_\omega$ be any countable almost disjoint family. Assume that $A_\beta$, $\sigma_\beta$ have been defined for $\beta < \alpha$ such that they satisfy the requirements stated above so far. We inductively define $\{C_s : s \in 2^{<\omega}\} \subseteq \H^-(\A_\alpha)$, $\{\eta_s : s \in 2^{<\omega}\} \subseteq 2^{<\frak{bs}}$ and ordinals $\{\alpha_s : s \in 2^{<\omega}\}$ as follows:
    \begin{enumerate}
        \item If $U_\alpha \in \I^+(\A_\alpha)$, let $C_\emptyset \in \H^-(\A_\alpha)$ such that $\c{C_\emptyset} \subseteq U_\alpha$. This is possible by Lemma \ref{lem:block.refinement.lemma}. If $U_\alpha \notin \I^+(\A_\alpha)$, let $C_\emptyset \in \H^-(\A_\alpha)$ be any block subspace.

        \item $\eta_s := \tau_{C_s}^{\A_\alpha}$ and $\alpha_s := \dom(\eta_s)$.

        \item $C_{s^\frown0},C_{s^\frown1} \in \H^-(\A_\alpha)$ are such that $\c{C_{s^\frown0}} \subseteq \c{C_s} \cap Z_{\alpha_s}^0$ and $\c{C_{s^\frown1}} \subseteq \c{C_s} \cap Z_{\alpha_s}^1$ (which both exist by properties (ii) and (iii) of Lemma \ref{lem:splitting.lemma}).
    \end{enumerate}
    We then define $\eta_f := \bigcup_{n<\omega} \eta_{f\restrictedto n}$ for any $f \in 2^\omega$, and note that $\eta_f \in 2^{<\frak{bs}}$ as $\frak{bs}$ has uncountable cofinality. We note the following observations:
    \begin{enumerate}
        \item If $s \sqsubseteq t$, then $C_s \leq C_t$ and $\eta_s \sqsubseteq \eta_t$.

        \item If $s$ and $t$ are incompatible, then so are $\eta_s$ and $\eta_t$.

        \item Consequently, $f \neq g$, then $\eta_f$ and $\eta_g$ are incompatible.
    \end{enumerate}
    Since $\alpha < \frak{c}$, we may let $f \in 2^\omega$ be such that $\eta_f \not\sqsubseteq \sigma_\beta$ for all limit ordinals $\beta < \alpha$. We now have a $\leq$-decreasing sequence $(C_{f\restrictedto n})_{n<\omega}$ in $\H^-(\A_\alpha)$, so by Proposition \ref{prop:H^-(A).is.(p)} there exists some $D \in \H^-(\A_\alpha)$ such that $D \leq^* C_{f\restrictedto n}$ for all $n$.

    Let $\sigma_{\alpha} := \eta_f$. Observe that for all $\xi < \dom(\eta_f)$, if $m < \omega$ is such that $\xi \in \dom(\eta_{f\restrictedto n})$, and $D/N \leq C_{f\restrictedto n}$, then:
    \begin{align*}
        \c{D/N} \cap Z_\xi^{1-\eta_f(\xi)} = \c{D/N} \cap Z_\xi^{1-\eta_{f\restrictedto n}(\xi)} \subseteq \c{C_{f\restrictedto n}} \cap Z_\xi^{1-\eta_{f\restrictedto n}(\xi)} \in \I(\A_\alpha),
    \end{align*}
    so that for all $\xi < \dom(\sigma_{\alpha})$, $\c{D} \cap Z_\xi^{1-\sigma_\alpha(\xi)} \in \I(\A)$ (as $\c{D} \setminus \c{D/N}$ is very small, by Lemma \ref{lem:ad.is.very.small}). Therefore, for each $\xi < \dom(\sigma_\alpha)$ we may find a finite subset $\F_\xi \subseteq \A_\alpha$ such that $\bb{\c{D} \cap Z_\xi^{1-\sigma_\alpha(\xi)}} \setminus \bigcup_{A \in \F_\xi} \c{A}$ is small. Now let:
    \begin{align*}
        \B_0 &:= \{A_{\beta+n} : \sigma_\beta \sqsubseteq \sigma_\alpha\}, \\
        \B_1 &:= \bigcup_{\xi < \dom(\sigma_\alpha)} \F_\xi, \\
        \B &:= \B_0 \cup \B_1.
    \end{align*}

    One may show, in the same way of proving Claim \ref{claim:B.smaller.than.bs}, that $|\B| < \frak{bs}$. By Lemma \ref{lem:block.subspaces.intersection.lemma}, the set:
    \begin{align*}
        \{\c{B'} \cap \c{D} : B' \in \B \text{ and } \c{B'} \cap \c{D} \text{ is infinite-dimensional}\}
    \end{align*}
    is an almost disjoint family of block subspaces of $\c{D}$. Since $D \in \H^-(\A_\alpha)$, we have that $\c{D}$ is not in this set. By Lemma \ref{lem:bs.leq.avec}, $|\B| < \frak{bs} \leq \avec$, so there exists some $A_\alpha \leq D$ such that $A_\alpha$ is almost disjoint from every element of $\B$. 

    \begin{claim}
        \hfill
        \begin{enumerate}
            \item For all $\xi \in \dom(\sigma_\alpha)$, $\c{A_\alpha} \setminus Z_\xi^{\sigma_\alpha(\xi)}$ is small.

            \item For all $\beta < \alpha$, $A_\alpha$ and $A_\beta$ are almost disjoint.
        \end{enumerate}
    \end{claim}

    \begin{midproof}
        \hfill
        \begin{enumerate}
            \item Note that $\c{A_\alpha} \setminus Z_\xi^{\sigma_\alpha(\xi)} = \c{A_\alpha} \cap Z_\xi^{1-\sigma_\alpha(\xi)}$. We have that $\c{A_\alpha} \cap Z_\xi^{1 - \sigma_\alpha(\xi)} \setminus \bigcup_{A \in \F_\xi} \c{A}$ is small by above, and since $A_\alpha$ is almost disjoint from every $A \in \F_\xi$, we have that:
            \begin{align*}
                \c{A_\alpha} \cap Z_\xi^{1 - \sigma_\alpha(\xi)} \cap \bigcup_{A \in \F_\xi} \c{A} &\subseteq \bigcup_{A \in \F_\xi} \c{A_\alpha} \cap \c{A}.
            \end{align*}
            By Lemma \ref{lem:ad.is.very.small}, $\c{A_\alpha} \cap \c{A}$ is very small for each $A \in \F_\xi$, so $\c{A_\alpha} \cap Z_\xi^{1-\sigma_\alpha(\xi)}$ is small. 

            \item Let $\beta < \alpha$. If $A_\beta \in \B_0$, then $A_\alpha$ is almost disjoint from $A_\beta$ by construction, so we are done. Now assume that $A_\beta \notin \B_0$. Let $\xi$ be the least ordinal such that $\sigma_\beta(\xi) \neq \sigma_\alpha(\xi)$. By induction hypothesis, we have that $W := \c{A_\beta} \setminus Z_\xi^{\sigma_\beta(\xi)} = \c{A_\beta} \setminus Z_\xi^{1-\sigma_\alpha(\xi)}$ is small (and hence very small). Thus, we have that for some small set $W'$:
            \begin{align*}
                \c{A_\alpha} \cap \c{A_\beta} &\subseteq W \cup (\c{A_\alpha} \cap Z_\xi^{1-\sigma_\alpha(\xi)}) \\
                &\subseteq W \cup W' \cup \bigcup_{A \in F_\xi} \c{A_\alpha} \cap \c{A},
            \end{align*}
            which is also small as all the components in the union are very small, as desired.
        \end{enumerate}
    \end{midproof}
    This completes the induction. If $U \in \I^+(\A)$, then $\I^+(\A_\alpha)$ for all $\alpha$, and there exists some $\alpha$ such that $U_\alpha = U$. This implies that $\c{A_\alpha} \subseteq U$, as desired. 
\end{proof}

\section{Topological Ramsey spaces and the abstract Mathias forcing}
\label{sec:TRS.and.abstract.mathias.forcing}
We dedicate this section to an overview of the axioms of topological Ramsey spaces, and the abstract Mathias forcing developed in \cite{DMN15}. 

\subsection{Topological Ramsey spaces} 
We recap the four axioms presented by Todor\v{c}evi\'{c} in \cite{T10}, which are sufficient conditions for a triple $(\cal{R},\leq,r)$ to be a topological Ramsey space. Here, $\cal{R}$ is a non-empty set, $\leq$ a quasi-order on $\cal{R}$, and $r : \cal{R} \times \omega \to \cal{AR}$ a surjective function. We also define a sequence of maps $r_n : \cal{R} \to \cal{AR}$ by $r_n(A) := r(A,n)$ for all $A \in \cal{R}$. Let $\cal{AR}_n \subseteq \cal{AR}$ be the image of $r_n$ (i.e. $a \in \cal{AR}_n$ iff $a = r_n(A)$ for some $A \in \cal{R}$). 

The four axioms are as follows:
\begin{enumerate}
    \item[\textbf{(A1)}] 
    \begin{enumerate}[label=(\arabic*)]
        \item $r_0(A) = \emptyset$ for all $A \in \cal{AR}$.

        \item $A \neq B$ implies $r_n(A) \neq r_n(B)$ for some $n$.

        \item $r_n(A) = r_m(B)$ implies $n = m$ and $r_k(A) = r_k(B)$ for all $k < n$.
    \end{enumerate}
    For each $a \in \cal{AR}$, let $\lh(a)$ denote the unique $n$ in which $a \in \cal{AR}_n$. By Axiom \textbf{A1}(3), this $n$ is well-defined. Given $a,b \in \cal{AR}$, we write $a \sqsubseteq b$ if there exists some $A \in \cal{R}$ such that $a = r_n(A)$ and $b = r_m(A)$ for some $n \leq m$.

    \item[\textbf{(A2)}]  There is a quasi-ordering $\leq_\fin$ on $\cal{AR}$ such that:
    \begin{enumerate}[label=(\arabic*)]
        \item $\{a \in \cal{AR} : a \leq_\fin b\}$ is finite for all $b \in \cal{AR}$.

        \item $A \leq B$ iff $\forall n \exists m [r_n(A) \leq_\fin r_m(B)]$.

        \item $\forall a,b \in \cal{AR}[a \sqsubseteq b \wedge b \leq_\fin c \to \exists d \sqsubseteq c[a \leq_\fin d]]$.
    \end{enumerate}

    \item[\textbf{(A3)}] We may define the \emph{Ellentuck neighbourhoods} as follows: For any $A \in \cal{R}$, $a \in \cal{AR}$ and $n \in \N$, we let:
    \begin{align*}
        [a,A] &:= \{B \in \cal{R} : B \leq A \wedge \exists n[r_n(B) = a]\}, \\
        [n,A] &:= [r_n(A),A].
    \end{align*}
    Then the depth function defined by, for $B \in \cal{R}$ and $a \in \cal{AR}$:
    \begin{align*}
        \depth_B(a) := 
        \begin{cases}
            \min\{n < \omega : a \leq_\fin r_n(B)\}, &\text{if such $n$ exists}, \\
            \infty, &\text{otherwise},
        \end{cases}
    \end{align*}
    satisfies the following:
    \begin{enumerate}[label=(\arabic*)]
        \item If $\depth_B(a) < \infty$, then for all $A \in [\depth_B(a),B]$, $[a,A] \neq \emptyset$.

        \item If $A \leq B$ and $[a,A] \neq \emptyset$, then there exists $A' \in [\depth_B(a),B]$ such that $\emptyset \neq [a,A'] \subseteq [a,A]$.
    \end{enumerate}
    For each $A \in \cal{R}$, we let:
    \begin{align*}
        \cal{AR}\restrictedto A := \{a \in \cal{AR} : \exists n[a \leq_\fin r_n(A)]\}.
    \end{align*}
    If $a \in \cal{AR}\restrictedto A$, we also define:
    \begin{align*}
        \cal{AR}\restrictedto[a,A] &:= \{b \in \cal{AR}\restrictedto A : a \sqsubseteq b\}, \\
        r_n[a,A] &:= \{b \in \cal{AR}\restrictedto[a,A] : \lh(b) = n\}.
    \end{align*}

    \item[\textbf{(A4)}] If $\depth_B(a) < \infty$ and if $\O \subseteq \cal{AR}_{\lh(a)+1}$, then there exists $A \in [\depth_B(a),B]$ such that $r_{\lh(a)+1}[a,A] \subseteq \O$ or $r_{\lh(a)+1}[a,A] \subseteq \O^c$.
\end{enumerate}

Note that by axiom \textbf{A1}, we may identify each element $A \in \cal{R}$ as a sequence of elements of $\cal{AR}$, via the map $A \mapsto (r_n(A))_{n<\omega}$. Therefore, we may identify $\cal{R}$ as a subset of $\cal{AR}^\N$. $\cal{AR}^\N$ may be equipped with a natural metric topology by considering the first difference metric --- that is, if $(a_n)_{n<\omega},(b_n)_{n<\omega} \in \cal{AR}^\N$, then:
\begin{align*}
    d((a_n)_{n<\omega},(b_n)_{n<\omega}) = 2^{-\min\{n : a_n \neq b_n\}}.
\end{align*}

\begin{definition}
    A triple $(\cal{R},\leq,r)$ is said to be a \emph{closed triple} if $\cal{R}$ is a metrically closed subset of $\cal{AR}^\N$. A closed triple $(\cal{R},\leq,r)$ is an \emph{\textbf{A2}-space} if $(\cal{R},\leq,r)$ satisfies the axioms \textbf{A1}-\textbf{A3}.
\end{definition}

An \textbf{A2}-space $(\cal{R},\leq,r)$ satisfying \textbf{A4} is a topological Ramsey space, as it satisfies the abstract Ellentuck theorem --- that is, a subset $\X \subseteq \cal{R}$ is Ramsey iff it has the property of Baire relative to the Ellentuck topology. This is the abstract Ellentuck theorem by Todor\v{c}evi\'{c} --- see Theorem 5.4 of \cite{T10}.

Mathias' original proof of his celebrated theorem in \cite{M77} (i.e. Theorem \ref{thm:mathias.thm.vector.spaces} for $[\omega]^\omega$) uses the (original) Mathias forcing localised to a coideal $\H$. As we shall use a similar approach to prove Theorem \ref{thm:mathias.thm.vector.spaces}, it is necessary to extend the notion of a coideal to topological Ramsey spaces. 

\begin{definition}[Definition 3.1, \cite{DMN15}]
\label{def:coideal}
    Let $(\cal{R},\leq,r)$ be a closed triple satisfying \textbf{A1}-\textbf{A4}. A set $\H \subseteq \cal{R}$ is a \emph{coideal} if it satisfies the following:
    \begin{enumerate}
        \item (Upward-Closure) If $A \in \H$ and $A \leq B$, then $B \in \H$.
        
        \item (\textbf{A3} mod $\H$) For all $A \in \H$ and $a \in \cal{AR}\restrictedto A$, we have that:
        \begin{enumerate}
            \item $[a,B] \cap \H \neq \emptyset$ for all $B \in [\depth_A(a),A] \cap \H$.
            
            \item If $B \in \H \restrictedto A$ and $[a,B] \neq \emptyset$, then there exists some $A' \in [\depth_A(a),A] \cap \H$ such that $\emptyset \neq [a,A'] \subseteq [a,B]$.
        \end{enumerate}
        
        \item (\textbf{A4} mod $\H$) Fix $A \in \H$ and $a \in \cal{AR} \restrictedto A$. For any $\O \subseteq \cal{AR}_{\lh(a)+1}$, there exists $B \in [\depth_A(a),A] \cap \H$ such that $r_{\lh(a)+1}[a,B] \subseteq \O$ or $r_{\lh(a)+1}[a,B] \subseteq \O^c$.
    \end{enumerate}
\end{definition}

A notion of semiselectivity for a coideal $\H$ was also given as Definition 3.4-3.7 in \cite{DMN15}, which says for all $A \in \H$ and a family of dense open sets, there exists some $B \in \H\restrictedto A$ that diagonalises the family. We shall propose a slightly stronger version of semiselectivity, which instead asks that for all $A \in \H$, $a \in \cal{AR}\restrictedto A$ and family of dense open sets, there exists some $B \in \H\restrictedto[a,A]$ such that $B$ diagonalises the family.

\begin{definition}
\label{def:semiselective}
    Let $(\cal{R},\leq,r)$ be a closed triple satisfying \textbf{A1}-\textbf{A4}, and let $\H \subseteq \cal{R}$ be a coideal. Let $A \in \H$ and $a \in \cal{AR}\restrictedto A$.
    \begin{enumerate}
        \item A family of subsets $\vec{\D} = \{\D_b\}_{b \in \cal{AR}\restrictedto[a,A]}$ is \emph{dense open below $[a,A]$ in $\H$} if for all $b \in \cal{AR}\restrictedto[a,A]$, $\D_b$ is a $\leq$-downward closed subset of $\H\restrictedto[b,A]$, and for all $B \in \H\restrictedto[b,A]$, there exists some $C \in \H\restrictedto[b,B]$ such that $C \in \D_b$.
        
        \item Let $\vec{\D} = \{\D_b\}_{b \in \cal{AR}\restrictedto[a,A]}$ be dense open below $[a,A]$ in $\H$. We say that $B \in \H\restrictedto[a,A]$ \emph{diagonalises} $\vec{\D}$ if for all $b \in \cal{AR}\restrictedto[a,A]$, there exists some $A_b \in \D_b$ such that $[b,B] \subseteq [b,A_b]$.

        \item We say that $\H$ is \emph{semiselective} if for all $A \in \H$, $a \in \cal{AR}\restrictedto A$, every dense open family below $[a,A]$ in $\H$ has a diagonalisation in $\H$.
    \end{enumerate}
\end{definition}

\begin{lemma}
    Let $(\cal{R},\leq,r)$ be a closed triple satisfying \textbf{A1}-\textbf{A4}. Then $\cal{R}$ is a semiselective coideal.
\end{lemma}

This is stated as Lemma 3.11 of \cite{Y24}. We replicate the proof here. 

\begin{proof}
    Note that $\cal{R}$ is a coideal by definition. Fix some $A \in \cal{R}$ and $a \in \cal{AR}\restrictedto A$. Suppose that $\vec{\D} = \{\D_b\}_{b \in \cal{AR}\restrictedto[a,A]}$ is dense open below $[a,A]$. We shall define a fusion sequence $(A_n)_{n<\omega}$ in $[a,A]$, with $a_{n+1} := r_{\lh(a)+n+1}(A_n)$, such that $A_{n+1} \in [a_{n+1},A_n]$: Let $A_0 := A$, and suppose that $A_n$ has been defined. Let $\{b_i : i < N\}$ enumerate the set of all $b \in \cal{AR}\restrictedto A_n$ such that $a \sqsubseteq b$ and $b \leq_\fin a_{n+1}$. Let $A_{n+1}^0 := A_n$. If $A_{n+1}^i \in [a_{n+1},A_n]$ has been defined, let $B_{n+1} \in \D_{b_i}$ be such that $B_{n+1} \in [b_i,A_{n+1}^i]$, which exists as $\D_{b_i}$ is dense open in $[b_i,A]$. By \textbf{A3}, we then let $A_{n+1}^{i+1} \in [a_{n+1},A_n^i]$ such that $[b_i,A_{n+1}^{i+1}] \subseteq [b_i,B_{n+1}]$. We complete the induction by letting $A_{n+1} := A_{n+1}^N$. Let $B$ be the limit of the fusion sequence $(A_n)_{n<\omega}$. Since $\cal{R}$ is metrically closed in $\cal{AR}^\N$, $B \in \cal{R}$ is well-defined, and we have that $B$ diagonalises $\vec{\D}$. 
\end{proof}

We will see that the semiselectivity property plays an important role in both combinatorial forcing and the abstract Mathias forcing.

\subsection{Combinatorial forcing}
Let $(\cal{R},\leq,r)$ be a closed triple satisfying \textbf{A1}-\textbf{A4}, and let $\H \subseteq \cal{R}$ be a coideal. We also fix a subset $\X \subseteq \cal{R}$. Unless stated otherwise, we assume that $\H$ is semiselective.

\begin{definition}
    Let $A \in \H$ and let $a \in \cal{AR}$. We say that: 
    \begin{enumerate}
        \item $A$ \emph{accepts} $a$ if $[a,A] \subseteq \X$.
        
        \item $A$ \emph{rejects} $a$ if $a \in \cal{AR}\restrictedto A$, and for all $B \in \H\restrictedto[a,A]$, $B$ does not accept $a$.
        
        \item $A$ \emph{decides} $a$ if $A$ accepts or rejects $a$.
    \end{enumerate}
\end{definition}

This localised variant of combinatorial forcing for topological Ramsey spaces was introduced in \cite{DMN15}, along with various basic properties of this combinatorial forcing. The details of the proofs of these properties with our definition of semiselectivity (Definition \ref{def:semiselective}) remains mostly unchanged, but we shall provide an overview for the sake of completion.

\begin{lemma}
\label{lem:combinatorial.forcing.basics}
    Let $A \in \H$ and let $a \in \cal{AR}$.
    \begin{enumerate}
        \item If $a \notin\cal{AR}\restrictedto A$, $A$ accepts $a$ (as $[a,A] = \emptyset$).
        
        \item If $A$ accepts $a$ and $B \in \H\restrictedto A$, then $B$ accepts $a$.
        
        \item If $A$ rejects $a$, $B \in \H\restrictedto A$ and $a \in \cal{AR}\restrictedto B$, then $B$ rejects $a$.
        
        \item If $A$ decides $a$ and $B \in \H\restrictedto A$, then $B$ decides $a$.
        
        \item If $a \in \cal{AR}\restrictedto A$, then there exists some $B \in \H\restrictedto[\depth_A(a),A]$ which decides $B$.
        
        \item If $a \in \cal{AR}\restrictedto A$ and $A$ decides $a$, then for all $B \in \H\restrictedto[\depth_A(a),A]$, $B$ decides $a$ in the same way $A$ does.
        
        \item $A$ accepts $a$ iff $A$ accepts every $b \in r_{\lh(a)+1}[a,A]$.
    \end{enumerate}
\end{lemma}

\begin{proof}
    The proofs are straightforward and mostly the same as the non-localised variant of combinatorial forcing --- we refer readers to Lemma 4.31 of \cite{T10}.
\end{proof}

\begin{lemma}
\label{lem:decisive.reduction}
    For all $A \in \H$ and $a \in \cal{AR}\restrictedto A$, then there exists some $B \in \H\restrictedto[a,A]$ which decides every $b \in \cal{AR}\restrictedto[a,A]$.
\end{lemma}

\begin{proof}
     For each $b \in \cal{AR}\restrictedto[a,A]$, let:
     \begin{align*}
        \D_b := \{B \in \H\restrictedto[b,A] : B \text{ decides } b\}.
     \end{align*}
     By Lemma \ref{lem:combinatorial.forcing.basics}(5), $\D_b$ is dense open in $\H\restrictedto[b,A]$. Since $\H$ is semiselective, there exists some $B \in \H\restrictedto[a,A]$ which diagonalises $\{\D_b\}_{b\in\cal{AR}\restrictedto[a,A]}$. If $b \in \cal{AR}\restrictedto B$, then $[b,B] \subseteq [b,A_b]$ for some $A_b \in \D_b$. Since $A_b$ decides $b$, $B \in \H\restrictedto A_b$ and $b \in \cal{AR}\restrictedto B$, by Lemma \ref{lem:combinatorial.forcing.basics}(2) and (3), $B$ decides $b$ as well. Note that if $b \notin \cal{AR}\restrictedto B$, then $[b,B] = \emptyset$, so $B$ trivially decides $b$.
\end{proof}

\begin{lemma}
\label{lem:one-point.reject.reduction}
    If $A$ rejects $a$, then there exists some $B \in \H\restrictedto[a,A]$ which rejects every $b \in r_{\lh(a)+1}[a,A]$.
\end{lemma}

\begin{proof}
    By Lemma \ref{lem:decisive.reduction}, we may let $A' \in \H[a,A]$ decide every $b \in \cal{AR}\restrictedto[a,A]$. Let:
    \begin{align*}
        \O := \{b \in r_{\lh(a)+1}[a,A] : A' \text{ rejects } b\}.
    \end{align*}
    By \textbf{A4} mod $\H$, there exists some $B \in \H\restrictedto[a,A]$ such that $r_{\lh(a)+1}[a,B] \subseteq \O$ or $r_{\lh(a)+1}[a,B] \subseteq \O^c$. It remains to show that the latter case is impossible. Since $A'$ decides every $b \in \cal{AR}\restrictedto A$, so does $B$. Therefore, if $r_{\lh(a)+1}[a,B] \subseteq \O^c$ then $B$ accepts every $b \in r_{\lh(a)+1}[a,B]$. By Lemma \ref{lem:combinatorial.forcing.basics}(7), this implies that $B$ accepts $a$, contradicting that $A$ rejects $a$.
\end{proof}

\begin{lemma}
\label{lem:completely.reject.reduction}
    If $A$ rejects $a$, then there exists some $B \in \H\restrictedto[a,A]$ which rejects every $b \in \cal{AR}\restrictedto[a,A]$. 
\end{lemma}

\begin{proof}
    Again, by Lemma \ref{lem:decisive.reduction} we may let $A' \in \H\restrictedto[a,A]$ decide every $b \in \cal{AR}\restrictedto[a,A]$. For each $b \in \cal{AR}\restrictedto[b,A]$, we let:
    \begin{align*}
        \D_b := \{B \in \H\restrictedto[b,A] : \text{$B$ accepts $b$ or $B$ rejects every $b' \in r_{\lh(b)+1}[b,A]$}\}.
    \end{align*}
    By Lemma \ref{lem:one-point.reject.reduction}, $\D_b$ is dense open in $\H\restrictedto[a,A]$, so by the semiselectivity property of $\H$ there exists some $B \in \H\restrictedto[a,A]$ which diagonalises $\{\D_b\}_{b\in\cal{AR}\restrictedto[a,A]}$. We shall induct on $\lh(b)$ and show that $B$ rejects every $b \in \cal{AR}\restrictedto[a,B]$ (if $b \in \cal{AR}\restrictedto[a,A]$ but $b \notin \cal{AR}\restrictedto[a,B]$, then $B$ trivially rejects $b$).

    For the base case $\lh(b) = \lh(a)$ (i.e. $b = a$), it follows from that $A$ rejects $a$, and that $B \in \H\restrictedto[a,A]$ (cf. Lemma \ref{lem:combinatorial.forcing.basics}(3)). If $B$ rejects some $b \in \cal{AR}\restrictedto[a,A]$ where $\lh(b) = n \geq \lh(a)$, then since $[b,B] \subseteq [b,A_b]$ for some $A_b \in \D_b$, we have that $B$ rejects every $b' \in r_{\lh(a)+1}[b,B]$, completing the induction step. Therefore, $B$ rejects every $b \in \cal{AR}\restrictedto[a,B]$.
\end{proof}

The abstract combinatorial forcing allows us to prove the abstract local Ellentuck theorem --- see Theorem 3.12 of \cite{DMN15}. We do not require the full strength of this result in this paper. Instead, we give an outline of the proof if $\H$ is semiselective, then every metrically open set $\X \subseteq \cal{R}$ is $\H$-Ramsey.

\begin{definition}
    A set $\X \subseteq \cal{R}$ is \emph{$\H$-Ramsey} if for all $A \in \cal{R}$ and $a \in \cal{AR}\restrictedto A$, there exists some $B \in [a,A]$ such that $[a,B] \subseteq \X$ or $[a,B] \subseteq \X^c$.
\end{definition}

\begin{proposition}
\label{prop:open.is.H-ramsey}
    Let $\H \subseteq \cal{R}$ be a semiselective coideal. If $\X \subseteq \cal{R}$ is metrically open, then $\X$ is $\H$-Ramsey.
\end{proposition}

\begin{proof}
    Let $A \in \H$ and let $a \in \cal{AR}\restrictedto A$. We consider combinatorial forcing w.r.t. $\X$. If $A$ does not reject $a$, then $[a,B] \subseteq \X$ for some $B \in \H\restrictedto[a,A]$, so we're done. Otherwise, $A$ rejects $a$, so by Lemma \ref{lem:completely.reject.reduction} there exists some $B \in \H\restrictedto[a,A]$ such that $B$ rejects every $b \in \cal{AR}\restrictedto[a,A]$. We claim that $[a,B] \subseteq \X^c$. Otherwise, we may let $C \in [a,B] \cap \X$. Since $\X$ is metrically open, there exists some $b \in \cal{AR}\restrictedto[a,B]$ such that $C \in [b] \subseteq \X$, where $[b] := \{D \in \cal{R} : b \sqsubseteq D\}$. In particular, we have that $[b,B] \subseteq \X$, contradicting that $B$ rejects $b$.
\end{proof}

\subsection{Abstract Mathias forcing}
The abstract Mathias forcing localised to a coideal $\H$ is defined as the forcing poset $(\mathbb{M}_\H,\leq)$, where:
\begin{align*}
    \mathbb{M}_\H = \{(a,A) : A \in \cal{R} \text{ and } a \in \cal{AR}\restrictedto A\},
\end{align*}
and that $(b,B) \leq (a,A)$ iff $[b,B] \subseteq [a,A]$. We also write $(b,B) \leq_0 (a,A)$ if $(b,B) \leq (a,A)$ and $b = a$. 

Recall that given a closed triple $(\cal{R},\leq,r)$ satisfying \textbf{A1}-\textbf{A4}, $\cal{R}$ is a closed subset of the product space $\cal{AR}^\N$, where $\cal{AR}$ is equipped with the discrete topology. Thus, there exists some absolute formula $\varphi$ such that:
\begin{align*}
    \cal{R} = \{A \in \cal{AR}^\N : \varphi(A)\}.
\end{align*}
Therefore, for any model $M$, we may write:
\begin{align*}
    \cal{R}^M := \{A \in \cal{AR}^\N : \varphi^M(A)\}.
\end{align*}

\begin{definition}
    An element $A \in \cal{R}$ is said to be $\mathbb{M}_\H$-generic (over a ground model $\V$) if:
    \begin{align*}
        G_A := \{(a,B) \in \mathbb{M}_\H : A \in [a,B]\}
    \end{align*}
    is a generic filter of $\mathbb{M}_\H$ (over $\V$).
\end{definition}

\begin{lemma}
    Let $\mathbb{M}_\H$ be the Mathias forcing localised to a coideal $\H$.
    \begin{enumerate}
        \item If $G$ is a generic filter of $\mathbb{M}_\H$, then there exists a unique $\mathbb{M}_\H$-generic $A_G \in \cal{R}$ such that $G = \{(a,A) \in \mathbb{M}_\H : A_G \in [a,A]\}$. 
        
        \item If $A \in \cal{R}$ is $M_\H$-generic, then $A_{G_A} = A$.
        
        \item If $G$ is a generic filter of $\mathbb{M}_\H$, then $G_{A_G} = G$.
    \end{enumerate}
\end{lemma}

\begin{proof}
    We only prove (1) --- (2) follows from the uniqueness of $A_G$ in (1), and (3) follow easily from the definitions. Let $\D_n := \{(a,A) \in \mathbb{M}_\H : \lh(a) \geq n\}$. We see that $\D_n$ is dense open in $\mathbb{M}_\H$ for all $n$ --- if $(a,A) \notin \D_n$ (i.e. $\lh(a) < n$) and $b \in r_n[a,A]$, then by \textbf{A3} mod $\H$ we have that $[b,A] \cap \H \neq \emptyset$. Let $B \in [b,A] \cap \H$, and we have that $[b,B] \subseteq [a,A]$, so $(b,B) \leq (a,A)$ and $(b,B) \in \D_n$. Therefore, $G \cap \D_n \neq \emptyset$ for all $n$, and since $G$ is a filter, for all $(a,A),(b,B) \in G$, $a \sqsubseteq b$ or $b \sqsubseteq a$. Thus, there exists a unique $A_G \in \cal{R}$ such that for all $n$, $(r_n(A_G),A) \in G$ for some $A$. It's easy to verify that $G = G_{A_G}$.
\end{proof} 

\begin{definition}
    Let $(\cal{R},\leq,r)$ be a closed triple satisfying \textbf{A1}-\textbf{A4}, and let $\H \subseteq \cal{R}$ be a coideal.
    \begin{enumerate}
        \item $\mathbb{M}_\H$ satisfies the \emph{Mathias property} if whenever $A \in \cal{R}$ is $\mathbb{M}_\H$-generic over a ground model $\V$ and $B \leq A$, $B$ is also $\mathbb{M}_\H$-generic over $\V$.
        
        \item $\mathbb{M}_\H$ satisfies the \emph{Prikry property} if whenever $\varphi$ is a sentence in the forcing language, then we have that for all $p \in \mathbb{M}_\H$, there exists some $q \leq_0 p$ such that $q \forces \varphi$ or $q \forces \neg\varphi$.
    \end{enumerate}
\end{definition}

Given a subset $\D \subseteq \mathbb{M}_\H$, we let:
\begin{align*}
    \bar{\D} &:= \bigcup_{(a,A) \in \D} [a,A].
\end{align*}
Note that $\bar{\D}$ is an Ellentuck-open subset of $\cal{R}$ in $\V$. Thus, if $\D$ is dense open in $\mathbb{M}_\H$, then for all $A \in \H$ and $a \in \cal{AR}\restrictedto A$, $A$ does not reject $a$ w.r.t. $\bar{\D}$ (otherwise, by Lemma \ref{lem:completely.reject.reduction} there exists some $B \in [a,A]$ which rejects every $b \in \cal{AR}\restrictedto[a,A]$, contradicting that $\D$ is dense in $\mathbb{M}_\H$). In other words, for all $A \in \H$ and $a \in \cal{AR}\restrictedto A$, there exists some $B \in \H\restrictedto[a,A]$ such that $[a,B] \subseteq \bar{\D}$.

\begin{notation}
    Given $a,b \in \cal{AR}$, we write:
    \begin{align*}
        b \preceq_\fin a \iff b \leq_\fin a \text{ and } \forall a' \sqsubseteq a \,(a' \neq a \to b \not\leq_\fin a').
    \end{align*}
\end{notation}

\begin{proposition}
\label{prop:semiselective.implies.Mathias}
    If $\H$ is semiselective, then $\mathbb{M}_\H$ has the Mathias property.
\end{proposition}

\begin{proof}
    Let $A_G$ be an $\mathbb{M}_\H$-generic element of $\cal{R}$ (over some ground model $\V$), and let $D \leq A_G$. We need to show that the filter: 
    \begin{align*}
        H := \{(a,A) \in \mathbb{M}_\H : D \in [a,A]\}
    \end{align*}
    is generic. Let $\D \subseteq \mathbb{M}_\H$ be a dense open subset (in the ground model $\V$). For each $a \in \cal{AR}$, we let:
    \begin{align*}
        \E_a := \{B \in \H : a \in \cal{AR}\restrictedto B \text{ and } \forall b \preceq_\fin a \, [b,B] \subseteq \bar{\D}\}.
    \end{align*} 

    \begin{claim}
        For all $A \in \H$ such that $a \in \cal{AR}\restrictedto A$, $\E_a$ is dense open in $\H\restrictedto[a,A]$.
    \end{claim}

    \begin{midproof}
        It's clear that $\E_a$ is open, so we shall show that it is dense in $\H\restrictedto[a,A]$. By the comment above, let $A' \in \H\restrictedto[a,A]$ such that $[a,A'] \subseteq \bar{\D}$. Let $\{b_i : 1 \leq i \leq N\}$ be an enumeration of all $b \preceq_\fin a$. Let $A_0 := A'$, and suppose that we have defined $A_i \in [a,A']$. We have that $b_{i+1} \in \cal{AR}\restrictedto A_i$ (as $b_{i+1} \leq_\fin a \sqsubseteq A_i$), so by the observation stated prior to the claim there exists some $A_{i+1}' \in \H\restrictedto[b_{i+1},A_i]$ which accepts $b_{i+1},A_i$ w.r.t. $\bar{\D}$, i.e. $[b_{i+1},A_{i+1}'] \subseteq \bar{\D}$. By \textbf{A3} mod $\H$, there exists some $A_{i+1} \in \H\restrictedto[a,A]$ such that $[a_{i+1},A_{i+1}] \subseteq [a_{i+1},A_{i+1}']$. This completes the induction, and $(a,A_N) \in \D$, as desired. 
    \end{midproof}

    Fix any $(a,A) \in G$. Since $\H$ is semiselective, for all $(b,B) \leq (a,A)$ in $\mathbb{M}_\H$, there exists some $C \in \H\restrictedto[b,B]$ which diagonalises $\vec{\E}_b := \{\E_c\}_{c\in\cal{AR}\restrictedto[b,A]}$ below $A$. That is, the set:
    \begin{align*}
        \bar{\D} := \{(b,B) \in \mathbb{M}_\H : \text{$B$ diagonalises $\vec{\E}_b$ below $A$}\}
    \end{align*}
    is dense open in $\mathbb{M}_\H$. Since $G$ is generic, we may let $(b,B) \in G \cap \D'$. Then $A_G \in [b,B]$, and for all $c \in \cal{AR}\restrictedto[b,A_G]$ and $d \preceq_\fin c$, $[d,A_G] \subseteq [d,B] \subseteq \bar{\D}$. Let $n$ be the least integer such that $\depth_{A_G}(r_n(D)) =: m \geq \lh(b)$. Then $r_n(D) \preceq_\fin r_m(A_G) \in \cal{AR}\restrictedto[b,B]$, so $D \in [r_n(D),A_G] \subseteq [r_n(D),B] \subseteq \bar{\D}$. This implies that $(r_n(D),B) \in H \cap \D$, and since $\D$ is arbitrary, $H$ is a generic filter as required.
\end{proof}

\begin{proposition}
\label{prop:semiselective.implies.Prikry}
    If $\H$ is semiselective, then $\mathbb{M}_\H$ has the Prikry property.
\end{proposition}

\begin{proof}
    Let $\varphi$ be a sentence in the forcing language. Let:
    \begin{align*}
        \D &:= \{q \in \mathbb{M}_\H : q \forces \varphi\}.
    \end{align*}
    We consider doing combinatorial forcing w.r.t. the set $\bar{\D}$. Let $(a,A) \in \mathbb{M}_\H$, and by Lemma \ref{lem:combinatorial.forcing.basics}(2) and (3) there exists some $B \in [a,A]$ which decides $a$. If $B$ accepts $a$, then $[a,B] \subseteq \bar{\D}$, so $(a,B) \in \D$ as $\D$ is open. In this case, $(a,B) \leq_0 (a,A)$ and $(a,B) \forces \varphi$. Otherwise, by Lemma \ref{lem:completely.reject.reduction} there exists some $B \in [a,A]$ which rejects every $b \in \cal{AR}\restrictedto[a,A]$. In other words, for all $(b,C) \leq (a,B)$, $(b,C) \notin \D_0$, i.e. $(b,C)$ does not force $\varphi$. But this precisely means that $(a,B) \forces \neg\varphi$, as needed.
\end{proof}

\section{Topological Ramsey space of block subspaces}
\label{sec:TRS.block.subspaces}
By Hindman's theorem, the pigeonhole principle holds for $E^{[\infty]}$ if $E$ is a vector space over $\mathbb{F}_2$. That is, for all $Y \subseteq E$ and $A \in E^{[\infty]}$, there exists some $B \leq A$ such that $\c{B} \subseteq Y$ or $\c{B} \subseteq Y^c$. Consequently, if for each $A := (x_n)_{n<\omega} \in E^{[\infty]}$ we let:
\begin{align*}
    r_n(A) := (x_0,\dots,x_{n-1}),
\end{align*}
then $(E^{[\infty]},\leq,r)$ is a closed triple satisfying \textbf{A1}-\textbf{A4}. We shall discuss the intersection between the local Ramsey theory of $(E^{[\infty]},\leq,r)$ as a topological Ramsey space, and the local Ramsey theory of block subspaces developed Smythe, followed by an application of the abstract Mathias forcing to prove Theorem \ref{thm:mathias.thm.vector.spaces}.

The symbols $<$ and $\leq$ will appear repeatedly throughout this section. To avoid confusion, the symbol $<$ shall be reserved for Notation \ref{not:notations}(5) and (6), while $\leq$ shall be reserved for the quasi-order $\leq$ in $(E^{[\infty]},\leq,r)$. 

\subsection{Semicoideals of block subspaces}
Readers may notice that there are two separate definitions for a coideal $\H \subseteq E^{[\infty]}$, i.e. Definition \ref{def:coideal.vector.spaces} and Definition \ref{def:coideal}. It turns out that these two definitions are equivalent. We remark that in this setting, if $A \in E^{[\infty]}$ and $a \in E^{[<\infty]}\restrictedto A$ (i.e. $a \in E^{[<\infty]}$ and $\c{a} \subseteq \c{A}$), we have that $\depth_A(a) = d_A(a)$ (see Notation \ref{not:notations}).

\begin{lemma}
    Let $\H \subseteq E^{[\infty]}$. The following are equivalent:
    \begin{enumerate}
        \item $\H$ is a semicoideal.
        
        \item $\H$ is $\leq$-upward closed and \textbf{A3} mod $\H$ holds.
    \end{enumerate}
\end{lemma}

\begin{proof}
    \underline{(2)$\implies$(1):} Suppose that $\H$ is $\leq$-upward closed and \textbf{A3} mod $\H$ holds. Let $A = (x_n)_{n<\omega} \in \H$ and $B \in E^{[\infty]}$ be such that $A \leq^* B$. Then there exists some $N$ such that $A/N \leq B$. Therefore, if we let $a := (x_N)$, then $a < C$ for all $C \in [a,C]$, so $[a,A] \subseteq [a,B]$. By \textbf{A3} mod $\H$ there exists some $C \in [a,A] \cap \H$. Then $C \in [a,B]$, and in particular $C \leq B$, so $B \in \H$ by $\leq$-upward closure. 
    
    \underline{(1)$\implies$(2):} Suppose that $\H$ is $\leq^*$-upward closed. Clearly, this implies that $\H$ is $\leq$-upward closed. We shall verify that \textbf{A3} mod $\H$ holds. Let $A \in \H$ and $a \in E^{[<\infty]}\restrictedto A$.
    \begin{enumerate}[label=(\alph*)]
        \item For any $B \in [d_A(a),A] \cap \H$, we have that $B \leq^* B/a$ (as $B/a \leq B/a$), so $B/a \in \H$. Since $B/a \leq a^\frown(B/a)$, we have that $a^\frown(B/a) \in \H$ by $\leq$-upward closure. But $a^\frown(B/a) \leq^* B$, so $a^\frown(B/a) \in [a,B] \cap \H$ as required.

        \item Suppose that $B \in \H\restrictedto A$ and $[a,B] \neq \emptyset$. Let $N := d_A(a)$, and let $M$ be large enough so that $B/M \leq A/N$. Then $C := r_N(A)^\frown(B/m) \in [d_A(a),A] \cap \H$ is well-defined, and:
        \begin{align*}
            D \in [a,C] &\implies a \sqsubseteq D \wedge D \leq C \\
            &\implies a \sqsubseteq D \wedge D/a \leq C/a \\
            &\implies a \sqsubseteq D \wedge D/a \leq B/M \\
            &\implies a \sqsubseteq D \wedge D/a \leq A/N \\
            &\implies D \in [a,A], 
        \end{align*}
        as desired.
    \end{enumerate}
\end{proof}

\begin{lemma}
    Let $\H \subseteq E^{[\infty]}$ be a semicoideal. Then $\H$ is a coideal (as in Definition \ref{def:coideal.vector.spaces}) iff \textbf{A4} mod $\H$ holds.
\end{lemma}

\begin{proof}
    Definition \ref{def:coideal.vector.spaces} is exactly the statement \textbf{A4} mod $\H$ for $a = \emptyset$, so it suffices to show that the converse holds. Let $A \in E^{[\infty]}$, $a \in E^{[<\infty]}\restrictedto A$, and let $\O \subseteq E^{[\lh(a)+1]}$. Define:
    \begin{align*}
        Y := \{x \in E : a < x \wedge a^\frown x \in \O\}.
    \end{align*}
    Since $\H$ is a coideal (as in Definition \ref{def:coideal.vector.spaces}), there exists some $B' \leq A$ such that $\c{B'} \subseteq Y$ or $\c{B'} \subseteq Y^c$. Let $M$ be large enough so that $B'/M \leq A/d_A(a)$, and note that $a < B'/M$. Let $B := B'/M$, and let $C := a^\frown B \in [a,A]$. Observe that:
    \begin{align*}
        b \in r_{\lh(a)+1}[a,C] \iff b = a^\frown x \text{ for some $x \in \c{B}$}.
    \end{align*}
    Therefore, if $\c{B} \subseteq Y$, then $r_{\lh(a)+1}[a,C] \subseteq \O$, and if $\c{B} \subseteq Y^c$, then $r_{\lh(a)+1}[a,C] \subseteq \O^c$. Thus, \textbf{A4} mod $\H$ holds.
\end{proof}

We now address the overlap in the selectivity properties. There are two natural questions:
\begin{enumerate}
    \item Is Definition \ref{def:semiselective} equivalent to the natural extension of semiselectivity of coideal of $[\omega]^\omega$ to semicoideal of block subspaces (i.e. Lemma 7.10, \cite{T10})?

    \item Is a selective (semi)coideal (as in Definition \ref{def:selective}) semiselective?
\end{enumerate}

We provide a positive answer to both of these problems.

\begin{lemma}
\label{lem:semiselectivity.equivalence}
    Let $\H \subseteq E^{[\infty]}$ be a semicoideal. The following are equivalent:
    \begin{enumerate}
        \item $\H$ is semiselective (as in Definition \ref{def:semiselective}).
        
        \item If $\{\D_n\}_{n<\omega}$ is a collection of dense open subsets of $\H$ (under $\leq$), then for all $A \in \H$ there exists some $B \in \H\restrictedto A$ such that $B/n \in \D_{d_A(r_n(B))}$ for all $n$.
    \end{enumerate}
\end{lemma}

In this statement (2) of the lemma, we also say that $B$ \emph{diagonalises $\{\D_n\}_{n<\omega}$ below $A$}.

\begin{proof}
    \underline{(1)$\implies$(2):} Let $\{\D_n\}_{n<\omega}$ be a collection of dense open subsets of $\H$, and let $A \in \H$. For each $b \in \cal{AR}\restrictedto A$, we let:
    \begin{align*}
        \D_b := \{b^\frown(B/b) : B \in \D_{d_A(b)} \text{ and } B \leq A\}.
    \end{align*}
    It's clear that $\D_b \subseteq \H\restrictedto[b,A]$ for all $b \in \cal{AR}\restrictedto A$, and that $\D_b$ is $\leq$-downward closed in $\H\restrictedto[b,A]$. We see that $\D_b$ is dense in $\H\restrictedto[b,A]$ --- if $B \in \H\restrictedto[b,A]$, then since $\D_{d_A(b)}$ is dense open in $\H\restrictedto A$, there exists some $C \in \D_{d_A(b)}$ such that $C \leq B$. Then $b^\frown(C/b) \in \D_b$ and $b^\frown(C/b) \leq B$. Since $\H$ is semiselective, there exists some $B \in \H\restrictedto A$ such that for all $b \in \cal{AR}\restrictedto B$, $[b,B] \subseteq [b,A_b]$ for some $A_b \in \D_b$. In particular, for all $n$ we have that $B \in [r_n(B),B] \subseteq [r_n(B),A_{r_n(B)}]$ for some $A_{r_n(B)} \in \D_{r_n(B)}$, so $B/n \leq A_{r_n(B)}/r_n(B) \in \D_{d_A(r_n(B))}$.

    \underline{(2)$\implies$(1):} Let $A \in \H$ and $a \in \cal{AR}\restrictedto A$. Let $\vec{\D} = \{\D_b\}_{b \in \cal{AR}\restrictedto A}$ be a dense open family below $[a,A]$. For each $b \in \cal{AR}\restrictedto[a,A]$, let:
    \begin{align*}
        \D_b' &:= \{B/b : B \in \D_b\}, \\
        \D_n &:= \bigcap_{\substack{b \in \cal{AR}\restrictedto[a,A] \\ d_A(b) = d_A(a) + n}} \D_b'.
    \end{align*}
    We claim that $\D_b'$ is dense open in $\H\restrictedto A$ for all $b \in \cal{AR}\restrictedto[a,A]$. This implies that $\D_n$ is dense open in $\H\restrictedto A$, as it is a finite intersection of dense open subsets of $\H\restrictedto A$. Let $B \in \H\restrictedto A$, and let $b \in \cal{AR}\restrictedto[a,A]$. Then $b^\frown(B/b) \in \H\restrictedto[b,A]$, and since $\D_b$ is dense open in $\H\restrictedto[b,A]$, there exists some $C \in \D_b$ such that $C \leq b^\frown(B/b)$. This implies that $C/b \leq B/b$, and $C/b \in \D_b'$, as desired.

    By our assumption, there exists some $B \in \H\restrictedto A$ such that $B/n \in \D_{d_A(r_n(B))}$ for all $n$. Observe that $a < B$, as $B \in \D_0 \subseteq \D_a$, so $B \leq C/a$ for some $C \in \D_a$. Therefore, $a^\frown B \in \cal{R}$ is well-defined. We claim that $a^\frown B$ diagonalises $\vec{\D}$. Let $b \in \cal{AR}\restrictedto[a,a^\frown B]$, and suppose that $d_A(b) = d_A(a) + n$. Let $b'$ be such that $b = a^\frown b'$, and note that $d_A(b') = d_A(b)$. Since $\c{b'} \subseteq \c{r_{d_B(b')}(B)}$, we have that $d_A(b') \leq d_A(r_{d_B(b')}(B))$. Therefore:
    \begin{align*}
        (a^\frown B)/b &= B/b' = B/d_B(b') \in \D_{d_A(r_{d_B(b')}(B))} \subseteq \D_{d_A(b')} = \D_{d_A(b)} \subseteq \D_b',
    \end{align*}
    so $A_b := b^\frown((a^\frown B)/b) \in \D_b$. Observe that $[b,B] = [b,A_b]$, and since $b$ is arbitrary, $B$ diagonalises $\vec{\D}$.
\end{proof}

\begin{lemma}
    Let $\H \subseteq E^{[\infty]}$ be a semicoideal. If $\H$ is selective, then $\H$ is semiselective.
\end{lemma}

\begin{proof}
    Let $\{\D_n\}_{n<\omega}$ be a family of dense open subsets of $\H$, and let $A \in \H$. Define a $\leq$-decreasing sequence in $\H$ as follows: Let $A_0 \in \D_0$ be such that $A_0 \leq A$. If $A_n \in \D_n$ has been defined, then let $A_{n+1} \in \D_{n+1}$ be such that $A_{n+1} \leq A_n$. By the selectivity of $\H$, there exists some $B \in \H\restrictedto A$ such that $B/n \leq A_{d_A(r_n(B))} \in \D_{d_A(r_n(B))}$ for all $n$. Since $\D_n$'s are all $\leq$-downward closed, $B/n \in \D_{d_A(r_n(B))}$ for all $n$.
\end{proof}

An abstract notion of selectivity was also proposed in \cite{DMN15}, but the definition appears to be problematic. For instance, Proposition 5.4 of \cite{DMN15} claims that every selective coideal is semiselective, but there appears to be a gap in the proof --- it is unclear how the element $A^{2,1} \in \D_{a_1^2}\restrictedto A^{1,n_1}$ may be obtained. It seems unlikely that this version of selectivity coincides with our definition of selectivity.

\begin{definition}
\label{def:filter}
    Let $\H \subseteq E^{[\infty]}$ be a semicoideal. 
    \begin{enumerate}
        \item $\H$ is a \emph{filter} if $(\H,\leq)$ is a filter (as a partial order), i.e. for all $A,B \in \H$ there exists some $C \in \H$ such that $C \leq A$ and $C \leq B$. 
            
        \item $\H$ is an \emph{ultrafilter} if it is a coideal and a filter.
    \end{enumerate}
\end{definition}

We shall dedicate the remaining portion of this subsection to showing that if $\U \subseteq E^{[\infty]}$ is a selective ultrafilter, then a block sequence $D \in (E^{[\infty]})^{\V[G]}$ is $\mathbb{M}_\U$-generic iff $D \leq^* A$ for all $A \in \U$. The proof is very similar to the case for the Ellentuck space --- a proof may be found in, for instance, Chapter 26 of \cite{H22}.

The following proposition, in the setting of topological Ramsey spaces, was stated as Theorem 6.6 of \cite{DMN15}, extending the original theorem in the setting $[\omega]^\omega$ in \cite{F98}. We shall reprove this result for block subspaces.

\begin{proposition}
\label{prop:semiselective.adds.selective.ultrafilter}
    Suppose that $E$ is a vector space over $\mathbb{F}_2$. Let $\H \subseteq E^{[\infty]}$ be a coideal. The following are equivalent:
    \begin{enumerate}
        \item $\H$ is semiselective.
        
        \item $(\H,\leq^*)$ is $\sigma$-distributive, and every $(\H,\leq^*)$-generic filter $\U$ is a selective ultrafilter.
    \end{enumerate}
\end{proposition}

\begin{proof}
    \underline{(1)$\implies$(2):} Let $\{\D_n\}_{n<\omega}$ be a family of dense open subsets of $\H$ under the partial order $\leq^*$. Note that each $\D_n$ is also dense open in $\H$ under $\leq$, so by semiselectivity there exists some $B \in \H$ which diagonalises $\{\D_n\}_{n<\omega}$ (below some $A \in \H$). Since $\D_n$ is $\leq^*$-downward closed, this implies that $B \in \D_n$ for all $n$, so $(\H,\leq^*)$ is $\sigma$-distributive.

    Let $\U$ be a $(\H,\leq^*)$-generic filter. It's clear that $\U$ is a filter (as in Definition \ref{def:filter}). We need to show that $\U$ is a coideal (and hence an ultrafilter). Let $Y \subseteq E$, and let $A \in \U \subseteq \H$. Since $Y$ is countable and $(\H,\leq^*)$ adds no reals, $Y$ is in the ground model. By Hindman's theorem, the set:
    \begin{align*}
        \D := \{B \in \H\restrictedto A : \c{B} \subseteq Y \text{ or } \c{B} \subseteq Y^c\}
    \end{align*}
    is dense open below $A$ in $(\H,\leq^*)$. Therefore, the genericity of $\U$ implies that $\U \cap \D \neq \emptyset$, i.e. there exists some $B \leq A$ in $\U$ such that $\c{B} \subseteq Y$ or $\c{B} \subseteq Y^c$. Hence, $\U$ is a coideal.

    We now show that $\U$ is selective. Let $(A_n)_{n<\omega}$ be a $\leq$-decreasing sequence in $\U$, and let $A := A_0$. Since $(\H,\leq^*)$ adds no reals, the sequence $(A_n)_{n<\omega}$ is in the ground model. For each $n$, let:
    \begin{align*}
        \D_n := \{B \in \H\restrictedto A : B \leq A_n \text{ or } (\nexists C \in \H\restrictedto A \text{ s.t. } C \leq B \text{ and } C \leq A_n)\}.
    \end{align*}
    It's clear that $\D_n$ is dense open below $A$ in $\H$ for all $n$. Since $\H$ is semiselective, we have that the set:
    \begin{align*}
        \E := \{B \in \H\restrictedto A : \text{$B$ diagonalises $\{\D_n\}_{n<\omega}$ below $A$}\}
    \end{align*}
    is dense open below $A$ in $\H$. By the genericity of $\U$, let $B \in \U \cap \E$. Then $B/n \in \D_{d_A(r_n(B))}$ for all $n$, and since $\U$ is a filter, $B$ is compatible with $A_n$ for all $n$ in $\U$, so $B/n \leq A_{d_A(r_n(B))}$ for all $n$. Thus, $B$ diagonalises $(A_n)_{n<\omega}$ below $A$.

    \underline{(2)$\implies$(1):} Let $\{\D_n\}_{n<\omega}$ be a family of dense open subsets of $\H$, and let $A \in \H$. Let $\U$ be a $(\H,\leq^*)$-generic filter containing $A$. We define a $\leq$-decreasing sequence $(A_n)_{n<\omega}$ in $\U$ as follows: By the genericity of $\U$, let $A_0 \in \U \cap \D_0$. Assume that $A_n \in \U \cap \D_n$ has been defined. Since $\D_{n+1} \cap \H\restrictedto A_n$ is dense open below $A_n$ in $\H$ and $A_n \in \U$, there exists some $A_{n+1} \in \U\restrictedto A_n$ such that $A_{n+1} \in \U \cap \D_{n+1}$. By the selectivity of $\U$, there exists some $B \in \U\restrictedto A$ such that $B/n \leq A_{d_A(r_n(B))} \in \D_{d_A(r_n(B))}$ for all $n$. Therefore, $B \in \H\restrictedto A$ diagonalises $\{\D_n\}_{n<\omega}$ below $A$.
\end{proof}

\begin{definition}
\label{def:capture.dense.set}
    Let $(\cal{R},\leq,r)$ be a closed triple satisfying \textbf{A1}-\textbf{A4}, and let $\H \subseteq \cal{R}$. Let $\D \subseteq \mathbb{M}_\H$. We say that $(a,A) \in \mathbb{M}_\H$ \emph{captures} $\D$ if for all $B \in \H\restrictedto[a,A]$, there exists some initial segment $b \in \cal{AR}\restrictedto[a,B]$ of $B$ such that $(b,B) \in \D$.
\end{definition}

We remark that if $\cal{AR}$ is countable (e.g. $E^{[<\infty]}$ is countable), then the statement ``$(a,A)$ captures $\D$'' is $\Pi_1$ with respect to the Polish topology of $\cal{R}$. Since $\cal{R}$ is a subspace of $\cal{AR}^\N$ and $\cal{AR}$ is equipped with the discrete topology, we may apply Mostowki's absoluteness theorem to conclude that ``$(a,A)$ captures $\D$'' is absolute.

\begin{lemma}
\label{lem:capturing.possible}
    Suppose that $E$ is a vector space over $\mathbb{F}_2$. Let $\U \subseteq E^{[\infty]}$ be a selective ultrafilter. Let $\D \subseteq \mathbb{M}_\U$ be dense open. For all $a \in E^{[<\infty]}$ and $A \in E^{[\infty]}$ such that $a \in E^{[<\infty]}\restrictedto A$, there exists some $B \in \U\restrictedto[a,A]$ such that $(a,B)$ captures $\D$.
\end{lemma}

\begin{proof}
    For each $b \in E^{[<\infty]}\restrictedto[a,A]$, let $A_b \in \U\restrictedto[b,A]$ be any block sequence such that $(b,A_b)\in \D$ if it exists --- otherwise, we let $A_b := A$. Since $\U$ is an ultrafilter, we may inductively define a $\leq$-decreasing sequence in $\U\restrictedto[a,A]$ such that $A_n \leq A_b/a$ for all $b \in E^{[<\infty]}\restrictedto[a,A]$ in which $d_A(b) \leq n$. Since $\U$ is selective, there exists some $B' \in \H\restrictedto(A/a)$ which diagonalises $(A_n)_{n<\omega}$ below $A$. Let $B := a^\frown B' \in \U\restrictedto[a,A]$, and consider the set:
    \begin{align*}
        \X := \{C \in E^{[\infty]} : C \in [a,B] \to (\exists b \in E^{[<\infty]}\restrictedto[a,C] \text{ s.t. } (b,B) \in \D)\}.
    \end{align*}
    Since $\X$ is a metrically open subset of $E^{[\infty]}$, by Proposition \ref{prop:open.is.H-ramsey} there exists some $C \in \U\restrictedto[a,B]$ such that $[a,C] \subseteq \X$ or $[a,C] \subseteq \X^c$. Observe that if $[a,C] \subseteq \X$, then $(a,C)$ captures $\D$, so we're done. It remains to show that $[a,C] \subseteq \X^c$ is not possible. Since $\D$ is dense in $\mathbb{M}_\H$, there exists some condition $(b,D) \leq (a,C)$ in $\D$. This means that $(b,A_b) \in \D$, and since $B/b \leq A_{d_A(b)} \leq A_b$, we have that $(b,B) \in \D$. In other words, there exists some $b \in E^{[<\infty]}\restrictedto C$ such that $(b,B) \in \D$, so $C \in \X$.
\end{proof}

\begin{proposition}
\label{prop:M-generic.characterisation}
    Suppose that $E$ is a vector space over $\mathbb{F}_2$. Let $\U \subseteq E^{[\infty]}$ be a selective ultrafilter. Then a block sequence $D$ is $\mathbb{M}_\U$-generic iff $D \leq^* A$ for all $A \in \U$.
\end{proposition}

\begin{proof}
    Given a block sequence $D$, we let:
    \begin{align*}
        G := \{(a,A) \in \mathbb{M}_\U : D \in [a,A]\}.
    \end{align*}

    \underline{$\implies$:} Suppose that $D$ is $\mathbb{M}_\U$-generic, i.e. $G$ is a generic filter. For any $A \in \U$ and $a \in E^{[<\infty]}\restrictedto A$, the set:
    \begin{align*}
        \D_A := \{(b,B) \in \mathbb{M}_\U : B \leq^* A \text{ and } b \in E^{[<\infty]}\restrictedto B\}
    \end{align*}
    is dense open in $\mathbb{M}_\U$ --- indeed, for any $(b,B) \in \mathbb{M}_\U$, let $C \in \U$ be such that $C \leq B$ and $C \leq A$. Then $(a,a^\frown(C/a)) \in \D_A$, and $C \leq^* B$ and $C \leq^* B$. Since $G$ is generic, $G \cap \D_A \neq \emptyset$, so we may let $(b,B) \in \D_A$ such that $D \in [b,B]$. This implies that $D \leq^* B \leq^* A$, as desired.

    \underline{$\impliedby$:} Suppose that $D \leq^* A$ for all $A \in \U$. Let $\D \subseteq \mathbb{M}_\U$ be dense open. To show that $D$ is generic, it suffices to show that $D \in [r_n(D),B]$ for some $(r_n(G),B) \in \D$. By Lemma \ref{lem:capturing.possible}, let $A := A_\emptyset \in \U$ be such that $(\emptyset,A)$ captures $\D$, and then for each $a \in E^{[<\infty]}$ we let $A_a \in \U$ be such that $(a,A_a)$ captures $\D$. We then inductively define a $\leq$-decreasing sequence $(A_n)_{n<\omega}$ in $\U$ where $A_n \leq A_a$ for all $d_A(a) \leq n$. Since $\U$ is selective, we may let $B \in \U\restrictedto A$ such that $B/a \leq A_{d_A(a)} \leq A_a$ for all $a \in E^{[<\infty]}\restrictedto B$. 

    We have that $D \leq^* B$, so let $a \in E^{[<\infty]}\restrictedto D \cap E^{[<\infty]} \cap B$  be such that $D/a \leq B/a \leq A_a$. Therefore, $(a,B)$ captures $\D$ in the ground model, so $(a,B)$ captures $\D$ in the generic extension (see the remark after Definition \ref{def:capture.dense.set}). Since $D \in [a,B]$, we have that $(r_n(D),B) \in \D$ for some $n$, as desired.
\end{proof}

\subsection{Proof of Theorem \ref{thm:mathias.thm.vector.spaces}}
Recall that for an uncountable cardinal $\kappa$, the forcing poset $\Coll(\omega,<\!\kappa)$ is the L\'{e}vy collapse, which collapses $\kappa$ to $\omega_1$. We shall follow the proof layout presented in \cite{NN18} (for the proof of Theorem 4.1). The following lemma is due to Solovay, and a proof may be found in \cite{solovay}. For the rest of this section, we fix a ground model $\V$ of $\ZFC$, and all vector spaces are over $\mathbb{F}_2$.

\begin{lemma}
\label{lem:solovay.factor.lemma}
    Let $\kappa$ be an inaccessible cardinal, and let $\P \in \V$ be a forcing poset of size $<\kappa$. Let $G$ be a $\Coll(\omega,<\!\kappa)$-generic filter. If $H \in \V[G]$ is $\P$-generic over $\V$, then there exists a filter $G^*$ which is $\Coll(\omega,<\!\kappa)$-generic over $\V[H]$, and $\V[H][G^*] = \V[G]$.
\end{lemma}

\begin{proof}[Proof of Theorem \ref{thm:mathias.thm.vector.spaces}]
    Let $\H \in \V[G]$ be a selective coideal in $E^{[\infty]}$, and let $\X \in \L(\R)^{\V[G]}$ be a subset of $E^{[\infty]}$. Since every set in $\L(\R)^{\V[G]}$ is ordinal-definable from some real in $\V[G]$, there exists a ternary formula $\phi$, some real $r \in \V[G]$ and a sequence of ordinals $\alpha$ such that for all $B \in (E^{[\infty]})^{\V[G]}$:
    \begin{align*}
        B \in \X \iff \V[G] \models \phi[B,r,\alpha].
    \end{align*}
    We need to show that for all $A \in \H$ and $a \in E^{[<\infty]}\restrictedto A$, there exists some $B \in \H\restrictedto[a,A]$ such that $[a,B] \subseteq \X$ or $[a,B] \subseteq \X^c$. We fix some $A \in \H$ and $a \in E^{[<\infty]}\restrictedto A$.

    Since $\kappa$ is Mahlo, we may fix some inaccessible $\lambda < \kappa$ such that $\H \cap \V[G\restrictedto\lambda] \in \V[G\restrictedto\lambda]$, and $\V[G\restrictedto\lambda] \models \H \cap \V[G\restrictedto\lambda]$ is a selective coideal, and denote $\bar{\H} := \H \cap \V[G\restrictedto\lambda]$. By the $\kappa$-chain condition of $\Coll(\omega,<\!\kappa)$, we may also assume that $A \in \bar{\H}$ and $r \in \V[G\restrictedto\lambda]$. Let $\phi^*$ be the formula such that $\phi^*(x,y,w)$ holds iff $\forces_{\Coll(\omega,<\kappa)} \phi(\dot{x},\check{y},\check{w})$. Let:
    \begin{align*}
        \bar{\X} := \{B \in E^{[\infty]} \cap \V[G\restrictedto\lambda] : V[G\restrictedto\lambda] \models \phi^*[B,r,\alpha]\}.
    \end{align*}
    Note that $\bar{\X} \in \V[G\restrictedto\lambda]$. Furthermore, since the L\'{e}vy collapse is homogeneous, we have that for all $B \in E^{[\infty]} \cap \V[G\restrictedto\lambda]$:
    \begin{align*}
        B \in \X &\iff \V[G] \models \phi[B,r,\alpha] \\
        &\iff \forces_{\Coll(\omega,<\kappa)} \phi[B,r,\alpha] \\
        &\iff B \in \bar{\X}.
    \end{align*}
    Thus, $\bar{\X} = \X \cap \V[G\restrictedto\lambda]$. 

    Let $\U$ be a $(\bar{\H},\leq^*)$-generic filter such that $A \in \U$. Note that by Proposition \ref{prop:semiselective.adds.selective.ultrafilter}, $\U$ is a selective ultrafilter. Since $(\bar{\H},\leq^*)$ adds no reals, we have that $E^{[\infty]} \cap \V[G\restrictedto\lambda][\U] = E^{[\infty]} \cap \V[G\restrictedto\lambda]$. This implies that $\bar{\X} = \X \cap \V[G\restrictedto\lambda][\U]$. Furthermore, we have that:
    \begin{align*}
        \V[G\restrictedto\lambda][\U][G\restrictedto[\lambda,\kappa)] = \V[G\restrictedto\lambda][G\restrictedto[\lambda,\kappa)][\U] = \V[G][\U].
    \end{align*}

    \begin{claim}
        Let $B \in E^{[\infty]} \cap \V[G][\U]$ be $\mathbb{M}_\U$-generic over $\V[G\restrictedto\lambda][\U]$. Then:
        \begin{align*}
            B \in \X \iff V[G\restrictedto\lambda][\U][B] \models \phi^*[B,r,\alpha].
        \end{align*}
    \end{claim}

    \begin{midproof}
        We work in $\V[G][\U]$. Since $\kappa$ is inaccessible, $\Coll(\omega,<\!\lambda) * \dot{\mathbb{M}}_\U$ is a poset in $\V$ of size $<\!\kappa$, so by Lemma \ref{lem:solovay.factor.lemma} applied to the model $\V[G\restrictedto\lambda][\U]$ there exists a filter $G^*$, $\Coll(\omega,<\!\kappa)$-generic over $\V[G\restrictedto\lambda][\U][B]$ such that:
        \begin{align*}
            \V[G][\U] = \V[G\restrictedto\lambda][\U][G\restrictedto[\lambda,\kappa)] = \V[G\restrictedto\lambda][\U][B][G^*].
        \end{align*}
        Therefore:
        \begin{align*}
            B \in \X &\iff \V[G][\U] \models \phi[B,r,\alpha] \\
            &\iff \;\forces \phi(\dot{B},\check{r},\check{\alpha}) \text{ in $\V[G\restrictedto\lambda][\U][B]$} \\
            &\iff \V[G\restrictedto\lambda][\U][B] \models \phi^*[B,r,\alpha],
        \end{align*}
        where the second equivalence follows from the homogeneity of the L\'{e}vy collapse.
    \end{midproof}
    
    We consider the forcing notion $\mathbb{M}_\U$ in $\V[G\restrictedto\lambda][\U]$, and note that $(a,A) \in \mathbb{M}_\U$. By the Prikry property of $\mathbb{M}_\U$ (Proposition \ref{prop:semiselective.implies.Prikry}), there exists some condition $(a,A') \leq_0 (a,A)$ in $\mathbb{M}_\U$ which decides the formula $\phi^*(\dot{A}_H,\check{r},\check{\alpha})$, where $\dot{A}_H$ is the canonical name for an $\mathbb{M}_\U$-generic element of $E^{[\infty]}$. Since $\V[G][\U] \models \U$ is countable and $\U$ is an ultrafilter, there exists some $A_H \in [a,A'] \cap \V[G]$ such that $A_H \leq^* B$ for all $B \in \U$. By Proposition \ref{prop:M-generic.characterisation}, $A_H$ is $\mathbb{M}_\U$-generic over $\V[G\restrictedto\lambda][\U]$. We shall finish the proof by showing that $[a,A_H] \subseteq \X$ or $[a,A_H] \subseteq \X^c$.

    Suppose that $(a,A') \forces \phi^*(\dot{A}_H,\check{r},\check{\alpha})$, and let $B \in [a,A_H]$. By the Mathias property of $\mathbb{M}_\U$, $B$ is also $\mathbb{M}_\U$-generic. Since $B \in [a,A']$, the condition $(a,A')$ is in the $\mathbb{M}_\U$-generic filter generated by $B$, so by the previous claim, $B \in \X$. Therefore, $[a,A_H] \subseteq \X$. Alternatively, if $(a,A') \forces \neg\phi^*(\dot{A}_H,\check{r},\check{\alpha})$, then $[a,A_H] \subseteq \X^c$.
\end{proof}

We may now prove Corollary \ref{cor:no.full.mad.family.in.L(R)}. Given an almost disjoint family $\A$ of subspaces, we let:
\begin{align*}
    \bar{\A} := \{B \in E^{[\infty]} : \c{B} \subseteq V \text{ for some } V \in \A\}.
\end{align*}
Observe that $\H^-(\A) \cap \bar{\A} = \emptyset$.

\begin{proof}[Proof of Corollary \ref{cor:no.full.mad.family.in.L(R)}]
    Suppose for a contradiction that $\A \in \L(\R)^{\V[G]}$ is a full mad family of subspaces. Note that $\bar{\A} \in \L(\R)^{\V[G]}$. By Proposition \ref{prop:H^-(A).is.selective}, $\H(\A)$ is a selective coideal, so by Theorem \ref{thm:mathias.thm.vector.spaces} we have that $\A$ is $\H$-Ramsey. Let $B \in \H(\A)$, and let $C \in \H(\A)\restrictedto A$ be such that $[\emptyset,C] \subseteq \bar{\A}$ or $[\emptyset,C] \subseteq \bar{\A}^c$. Since $C \in \H(\A)$ and $\H(\A) \cap \bar{\A} = \emptyset$, the first case is not possible. On the other hand, since $\A$ is mad, there exists some $V \in \A$ which is compatible with $C$, i.e. there exists some $D \in E^{[\infty]}$ such that $D \leq C$ and $D \leq A$. Then $D \in [\emptyset,C] \cap \bar{\A}$, another contradiction.
\end{proof}

We conclude this section by addressing the difficulty in removing the additional large cardinal hypothesis on $\kappa$. In \cite{NN18}, Neeman--Norwood successfully reduced the required large cardinal strength of $\kappa$ in the original Mathias' theorem \cite{M77} from being Mahlo to just inaccessible, with the additional requirement that the selective coideal $\H$ lies in $\L(\R)^{\V[G]}$. Their approach is as follows:
\begin{enumerate}
    \item For an ultrafilter $\U$ on $\omega$, they defined the \emph{diagonalisation forcing} $\P_\U$, and show that $\P_\U$ satisfies the Prikry and Mathias properties (where $\U$ need not be selective).
    
    \item Let $\Coll(\omega,<\!\kappa)$ be the L\'{e}vy collapse of an inaccessible cardinal, and let $G$ be a $\Coll(\omega,<\!\kappa)$-generic filter. There exists some $\beta < \kappa$ such that $\bar{\H} = \H \cap \V[G\restrictedto\beta] \in \V[G\restrictedto\beta]$. $\bar{\H}$ is a coideal on $\omega$, but it may not be selective. Let $\U \subseteq \bar{\H}$ be any ultrafilter on $\omega$ in $\V[G\restrictedto\beta]$.
    
    \item Repeat the argument of Mathias with Mathias forcing $\mathbb{M}_\U$ replaced by diagonalisation forcing $\P_\U$.
\end{enumerate}
Using the definition of ultrafilter for infinite block sequences (see Definition \ref{def:filter}), it is possible to define a version of diagonalisation forcing for infinite block sequences, and it does satisfy the Prikry and Mathias properties. The issue lies in the second step --- if $\H$ is a selective coideal in $E^{[\infty]}$ in $\V[G]$ (or $\L[G]^{\V[G]}$), it is unclear if $\bar{\H} = \H \cap \V[G\restrictedto\beta]$ is a coideal in $E^{[\infty]}$. It is even less clear if we may obtain an ultrafilter $\U \subseteq \bar{\H}$ in $\V[G\restrictedto\beta]$ --- the only current known construction of ultrafilters in $E^{[\infty]}$ assumes set theoretic hypotheses such as $\CH$ or $\MA$, and define an ultrafilter as the $\leq$-upward closure of a $\leq^*$-decreasing chain of block sequences of size continuum (see Theorem 5.3 of \cite{S18} and \cite{B87}). This requires one to repeatedly obtain weak diagonalisations (as in Definition \ref{def:(p)-semicoideal}) at countable limit ordinal stages of the induction. However, as $\bar{\H}$ need not satisfy the (p)-property, this construction may not be performed inside $\bar{\H}$ to obtain an ultrafilter $\U \subseteq \bar{\H}$. 

\section{Further remarks and open problems}
\label{sec:conclusion}
Although Theorem \ref{thm:main.full.mad.family} asserts the existence of a full mad family of subspaces in $\ZFC$, it remains unknown if every mad family of subspaces in $\ZFC$ is full.

\begin{prob}
    Is there a non-full mad family of subspaces?
\end{prob}

While there are various results about the cardinal $\avec$, the cardinal $\avec^*$ is not well-understood. In particular, while it is known that $\non(\M) \leq \avec$ holds in $\ZFC$, the $\avec^*$ variant of the problem remains unsolved.

\begin{prob}
    Is $\non(\M) \leq \avec^*$ true in $\ZFC$?
\end{prob}

Most of the current approaches to the non-existence of mad families in $[\omega]^\omega$ employ some Ramsey theoretic techniques (such as \cite{M77} \cite{NN18} \cite{ST19}). Consequently, it is difficult to use these approaches to prove the non-existence of mad families of block subspaces over $\mathbb{F} \neq \mathbb{F}_2$. Perhaps the approach highlighted in \cite{T18} may be used to answer the following question:

\begin{prob}
    Let $\kappa$ be an inaccessible cardinal, and let $G$ be $\Coll(\omega,<\!\kappa)$-generic. Are there mad families of subspaces (over any countable field) in $\L(\R)^{\V[G]}$? What if $\kappa$ is Mahlo?
\end{prob}

\section*{Acknowledgements}
The author would like to express his deepest gratitude to Spencer Unger and Iian Smythe for their helpful comments, suggestions and corrections.

\printbibliography[heading=bibintoc,title={References}]

\end{document}

%% file: preamble.tex
\usepackage{amsmath, amsthm, amssymb, amsfonts}
\usepackage{bbm}
\usepackage{bm}
\usepackage{graphicx}
\usepackage{xcolor}
\usepackage{enumitem}
\usepackage{tikz}
\usepackage{tikz-cd}
\usepackage{pgfplots}
\usepackage{rotating}
\usepackage{hyperref}
\usepackage{cleveref}

\pgfplotsset{compat=1.18}

\theoremstyle{plain}

\usepackage[backend=biber,sorting=nty,giveninits=true]{biblatex}

\numberwithin{equation}{section}

\newtheorem{theorem}{Theorem}[section]
\newtheorem{lemma}[theorem]{Lemma}
\newtheorem{proposition}[theorem]{Proposition}
\newtheorem{corollary}[theorem]{Corollary}
\newtheorem{claim}{Claim}

\theoremstyle{definition}
\newtheorem{definition}[theorem]{Definition}
\newtheorem{notation}[theorem]{Notation}

\newtheorem{prob}[theorem]{Problem}

\newcommand{\func}{\operatorname}

\renewcommand{\c}[1]{\langle #1 \rangle}
\newcommand{\bigc}[1]{\left\langle #1 \right\rangle}

\newcommand{\bb}[1]{\left( #1 \right)}
\renewcommand{\ss}[1]{\left\{ #1 \right\}}
\renewcommand{\mod}[1]{\left| #1 \right|}
\newcommand{\R}{\mathbb{R}}

\renewcommand{\P}{\mathbb{P}}
\newcommand{\Po}{\cal{P}}
\newcommand{\N}{\mathbb{N}}
\newcommand{\F}{\cal{F}}

\renewcommand{\bar}[1]{\overline{#1}}

\newcommand{\iffbreak}{\phantom{{} \iff {}}}

\newcommand{\cal}{\mathcal}
\renewcommand{\frak}{\mathfrak}

\renewcommand{\bf}{\mathbf}

\newcommand{\supp}{\func{supp}}

\newcommand{\m}[1]{\begin{matrix} #1 \end{matrix}}

\newcommand{\ZFC}{\mathsf{ZFC}}

\newcommand{\CH}{\mathsf{CH}}

\newcommand{\MA}{\mathsf{MA}}

\newcommand{\dom}{\func{dom}}

\newcommand{\A}{\cal{A}}
\newcommand{\B}{\cal{B}}

\newcommand{\non}{\func{non}}

\newcommand{\forces}{\Vdash}

\newcommand{\restrictedto}{\mathord{\upharpoonright}}

\newcommand{\Coll}{\func{Coll}}

\newcommand{\M}{\mathcal{M}}

\newcommand{\FIN}{\bf{FIN}}
\newcommand{\U}{\cal{U}}
\newcommand{\V}{\bf{V}}

\newcommand{\X}{\cal{X}}

\renewcommand{\H}{\cal{H}}
\newcommand{\I}{\cal{I}}
\newcommand{\lh}{\func{lh}}

\renewcommand{\O}{\cal{O}}
\newcommand{\depth}{\func{depth}}

\newcommand{\fin}{\mathrm{fin}}
\newcommand{\D}{\cal{D}}

\newcommand{\E}{\cal{E}}

\newcommand{\Lim}{\func{Lim}}
\newcommand{\avec}{\frak{a}_{\mathrm{vec},\mathbb{F}}}

\let\strokeL\L
\DeclareRobustCommand{\L}{\ifmmode\mathbf{L}\else\strokeL\fi}

\let\Sec\S
\renewcommand{\S}{\mathcal{S}}

\newenvironment{midproof}[1][Proof]
  {\begin{proof}[#1]}
  {\end{proof}}

\makeatletter
\DeclareRobustCommand\bigop[1]{%
  \mathop{\vphantom{\sum}\mathpalette\bigop@{#1}}\slimits@
}
\newcommand{\bigop@}[2]{%
  \vcenter{%
    \sbox\z@{$#1\sum$}%
    \hbox{\resizebox{\ifx#1\displaystyle.9\fi\dimexpr\ht\z@+\dp\z@}{!}{$\m@th#2$}}%
  }%
}
\makeatother

\DeclareFontShape{OT1}{cmr}{bx}{sc}{<-> cmbcsc10}{}